\documentclass{CSML}
\pdfoutput=1

\usepackage{lastpage}

\def\dOi{13(3:34)2017}
\lmcsheading%
{\dOi}
{1--\pageref{LastPage}}
{}
{}
{Mar.~11, 2016}
{Sep.~28, 2017}
{}


\usepackage[hidelinks]{hyperref}
\theoremstyle{plain} 

\newcommand{\m}{\mu}
\newcommand{\mut}{\tilde{\mu}}
\newcommand{\ma}{\mathcal}

\newcommand{\lmp}{\l \mu \mu'}

\newcommand{\lbe}{\l_{\b \pi \eta}}

\newcommand {\ls}{{\lambda}^{\tiny{\textit{Sym}}}_{\tiny{\textit{Prop}}}}

\newcommand{\lmt}{\overline{\lambda}\mu\tilde{\mu}}
\newcommand{\lmts}{\overline{\lambda}\mu\tilde{\mu}^*}

\def\a{\alpha}
\def\b{\beta}
\def\bb{\b^{\bot}}
\def\pib{\pi^{\bot}}

\def\l{\lambda}

\def\r{\rho}

\def\lm{\l \mu}
\def\lmp{\l \mu \mu'}

\def\etb{\eta_\bot}
\def\wt{\widetilde}

\def\ra{\rightarrow}
\def\raa{\hookrightarrow}

\def\reb{\ra_\b}
\def\rebb{\ra_{\b^{\bot}}}
\def\ree{\ra_e}
\def\rebn{\ra_{\b_0}}

\def\rep{\ra_\pi}

\def\sta{\star}
\def\lan{\langle}
\def\ran{\rangle}
\def\lfl{\lfloor}
\def\rfl{\rfloor}
\def\mf{\mathfrak}
\def\si{\sigma}
\def\r{\rho}
\def\t{\tau}

\def\v{\vdash}
\def\G{\Gamma}
\def\S{\Sigma}
\def\D{\triangle}
\def\F{\displaystyle\frac}
\def\vv{\rhd}

\begin{document}

\title[Strong normalization of $\ls$- and $\lmts$-calculi]{Strong normalization of $\ls$- and $\lmts$-calculi}

\author[P.~Batty\'anyi]{P\'eter Batty\'anyi\rsuper a}	
\address{{\lsuper a}Department of Computer Science, Faculty of Informatics,
University of Debrecen, Kassai \'ut 26, 4028 Debrecen, Hungary}	
\email{battyanyi.peter@inf.unideb.hu}  

\author[K.~Nour]{Karim Nour\rsuper b}	
\address{{\lsuper b}LAMA - \'Equipe LIMD,
Universit\'e Savoie Mont Blanc,
73376 Le Bourget du Lac}	
\email{karim.nour@univ-smb.fr}  





\begin{abstract}
  \noindent In this paper we give an arithmetical proof of the strong normalization of $\ls$ of Berardi and Barbanera
\cite{Ber-Bar}, which can be considered as a formulae-as-types translation of classical propositional
logic in natural deduction style.
Then we give a translation between the $\ls$-calculus and the $\lmts$-calculus, which is the implicational part of
the $\lmt$-calculus invented by Curien and Herbelin \cite{Cur-Her} extended with negation.
In this paper we adapt the method of David and Nour \cite{Dav-Nou3} for proving strong normalization.
The novelty in our proof is the notion of zoom-in sequences of redexes, which leads us directly to the proof of the main theorem.
\end{abstract}

\maketitle

\section*{Introduction}\label{S:one}

  It was revealed by the works of Murthy \cite{Mur} and Griffin \cite{Gri} that the Curry-Howard isomorphism,
{which establishes a correspondence between natural deduction style proofs in intuitionistic logic and terms
of the typed $\l$-calculus},  can be extended to the case of classical logic, as well. Since their discovery many
calculi appeared aiming to give an encoding of proofs formulated either
in classical natural deduction or in classical sequent calculus.

The $\lm$-calculus presented by Parigot in \cite{Par5} finds its origin in the so called Free Deduction (FD).
Parigot resolves the deterministic nature of intuitionistic natural deduction: unlike in the case of intuitionistic natural
deduction, when eliminating an instance of a cut in FD, there can be several choices for picking out the subdeductions to be
transformed. By introducing variables of a new kind, the so called $\m$-variables, {Parigot distinguishes formulas that}
are not active at the moment but the current continuation can be passed over to them. Besides the usual $\b$-reduction,
Parigot introduces a new reduction rule called the $\m$-rule corresponding to structural cut eliminations made necessary
by the occurrence of new forms of cuts due to the rule in connection with the $\m$-variables. The result is a calculus,
the $\lm$-calculus (Parigot \cite{Par1}), which is in relation with classical natural deduction.
The $\m'$-rule is the symmetric counterpart of the $\m$-rule. It was introduced by Parigot \cite{Par2} with the intention
of keeping the unicity of representation of data (Nour \cite{Nou1}), {the price was, however, that confluence had been lost}.
In the presence of other simplification rules besides $\m$ and $\m'$, even the strong normalization property is lost 
(Batty\'anyi \cite{Batt}).

Historically, the first calculus reflecting the symmetry of classical propositional logic was the $\ls$-calculus of Berardi
and Barbanera \cite{Ber-Bar} establishing a formulae-as-types connection with natural deduction in classical logic.
The calculus $\ls$ {uses an involutive negation which is not defined as} $A\ra \bot$. {There are negated and non-negated atomic types},
and the main connective is not the arrow but the classical $\wedge$ and $\vee$. Berardi and Barbanera make use of the natural symmetry 
of classical logic expressed by the de Morgan laws in defining negated types. In their paper,
Berardi and Barbanera proved that $\ls$ is strongly normalizable with a symmetric version of the Tait-Girard reducibility
method (Tait \cite{Tai}). In this paper, leaning on the combinatorial proof applied by David and Nour in \cite{Dav-Nou4},
we prove that $\ls$ is strongly normalizing. The novelty in our proof is the application of so-called zoom-in sequences of redexes,
which was inspired by the work in Raamsdonk et al. \cite{Sor}. We prove strong normalizability by verifying that it is closed under 
substitution.
From the assumption that $U[x:=V]$ is strongly normalizing and $U$, $V$ are strongly normalizing, we can identify a subterm $U'$
of a reduct of $U$ such that $U'[x:=V]$ also is strongly normalizing. The reduction sequence leading to $U'$ is a so-called
zoom-in sequence of redexes: each subsequent element is a subterm of the one-step reduct of the preceding one. We prove that
zoom-in sequences have useful invariant properties,
which makes it relatively easy for us to set the stage for the main theorem.
Due to its intrinsic symmetry in dealing with the typing relation, the $\ls$-calculus also proves to be very close to the calculus
named by Nour as classical combinatory logic (CCL). Nour \cite{Nou} defined a calculus of combinators which is equivalent to the
full classical propositional logic in natural deduction style. Then a translation is given in both directions between $\ls$ and CCL.

Curien and Herbelin introduced the $\lmt$-calculus (Curien et al. \cite{Cur-Her}), which established a correspondence, via the 
Curry-Howard isomorphism, between classical Gentzen-style sequent calculus {and a logical calculus}. The $\lmt$-calculus
possesses a rather strong symmetry: it has right-hand side and left-hand side terms (also referred to as environments).
The strong normalization of the calculus was proved by Polonovski \cite{Poi}, and a proof formalizable in first order Peano
arithmetic was found{ by David and Nour \cite{Dav-Nou3}.}

As to the connection between the $\lm$ and the $\lmt$-calculus, Curien and Herbelin \cite{Cur-Her} defined a translation both
for the call-by-value and the call-by-name part of the $\l \mu$-calculus into the $\lmt$-calculus. Rocheteau \cite{Roc} finished
this work by defining
simulations between the two calculi in both directions. In this paper we define the $\lmts$-calculus, which is the $\lmt$-calculus
extended with negation, and we describe translations between the $\lmts$-calculus and the $\ls$-calculus. As a consequence,
 we obtain that, if one of the calculi is strongly normalizable, then the other one necessary admits this property.

The proof applied in the paper is an adaptation of that of David and Nour \cite{Dav-Nou3}. David and Nour \cite{Dav-Nou3} gave 
arithmetic proofs, that is, proofs formalizable in first-order Peano arithmetic, for the strong normalizability of the $\lmt$- and 
Parigot's symmetric $\lm$-calculus. It is demonstrated that the set of strongly normalizable terms are closed under substitution.
The goal is achieved by applying implicitly an alternating substitution to find out which part of the substitution would be
responsible for being not strongly normalizable provided the basis of the substitution and the terms written in are strongly
normalizable. In this paper we reach the same goal by identifying a minimal non strongly normalizing sequence of redexes provided
an infinite reduction sequence is given. We call this sequence of redexes a minimal zoom-in reduction sequence. The idea of zoom-in
sequence was inspired by Raamsdonk et al. \cite{Sor}, where perpetual reduction strategies are defined in order to locate the minimal 
non strongly normalizing subterms of the elements of an infinite reduction sequence. Again, alternating substitutions are defined 
inductively starting from two sets of terms, and it is proven that zoom-in reduction sequences do not lead out of these substitutions.
With this in hand, the method of David and Nour \cite{Dav-Nou3} can be applied.

We prove the strong normalization of the \smash{$\ls$}-calculus, though our proof works with some slight modification in the
case of the $\lmt$-calculus, as well (Batty\'anyi \cite{Batt}). However, instead of repeating the proof here, we give a translation of
the \smash{$\ls$}-calculus into the $\lmt$-calculus, and vice versa. In fact, to make the connection more visible we define the
 $\lmts$-calculus, which is the $\lmt$-calculus extended with terms expressing negated types. Hence, we also obtain a new proof
of strong normalization of the $\lmt$-calculus.

The paper is organized as follows. In the first section we introduce the \smash{$\ls$}-calculus of Berardi and Barbanera,
and, as the first step towards strong normalization, prove that the permutation rules can be postponed. In the next section
we show that the $\b$, $\b^\bot$, $\pi$ and $\pi^\bot$ rules together are strongly normalizing. Section \ref{section3} introduces
the $\lmt$-calculus defined by Curien and Herbelin, and we augment the calculus with negation in order to make the comparison
of the \smash{$\ls$}- and the $\lmt$-calculi simpler. Section \ref{section4} provides translations between the \smash{$\ls$}- and
the $\lmts$-calculi such that the strong normalization of one of the calculi implies that of the other.
The last section contains conclusions with regard to the results of the paper.

\section{The $\ls$-calculus}

The ${\lambda }^{{\tiny\textit{Sym}}}$-calculus was
introduced by Berardi and Barbanera \cite{Ber-Bar}. It is
organized entirely around the duality in classical logic.
It has a negation ``built-in'': the negation of $A$ is not
defined as $A\ra\bot$. Each type is rather related to its
natural negated type by the notion of duality
introduced by negation in classical logic. In fact, Berardi and
Barbanera defined a calculus equivalent to first order Peano
arithmetic. However, we only consider here its propositional part,
denoted by $\ls$, since all the other calculi treated by us in
this work are concerned with propositional logic.

\begin{defi}\label{int:typels}
The set of types are built from two sets of base
types
$\mathcal{A}=\{ a,b,\ldots \}$ (atomic types) and
${\mathcal A}^{\bot} =\{ a^{\bot},b^{\bot},\ldots \}$ (negated
atomic types).

\begin{enumerate}
\item The set of m-types is defined by the following grammar
$$A := \a \mid {\a}^{\bot} \mid A\wedge A\mid A\vee A$$
where $\a $ ranges over $\mathcal{A}$ and ${\a}^{\bot}$ over
${\mathcal A}^{\bot}$.

\item The set of types is defined by the following grammar
$$C := A\mid \bot.$$
\item We define the negation of an m-type as follows

\begin{tabular}{ l l l l}
$(\a )^{\bot}=\a ^\bot \;$   & $(\a^{\bot})^{\bot}=\a\;$ &
$(A\wedge B)^{\bot}=A^{\bot}\vee B^{\bot}\;$   & $(A\vee B)^{\bot}=A^{\bot}\wedge B^{\bot}.$
 \end{tabular}\\
In this way we get an involutive negation,
i.e. for every $m$-type $A$, $(A^{\bot})^{\bot}=A$.
\newpage

\item The complexity of a type is defined inductively as follows.
\begin{itemize}
\item[] $cxty(A)=0$, if $A \in \mathcal{A} \cup \mathcal{A}^{\bot} \cup \{ \bot \}$.
\item[] $cxty(A_{1}\wedge A_{2})=cxty(A_{1}\vee A_{2})=cxty(A_{1})+cxty(A_{2})+1$.
\end{itemize}
Then, for every $m$-type $A$, $cxty(A)=cxty(A^{\bot})$.
\end{enumerate}

\end{defi}

\begin{defi}\hfill
\begin{enumerate}

\item We denote by $Var$ the set of term-variables.
The set of terms ${\mathcal T}$ of the $\ls$-calculus together with their typing rules are defined as
follows. In the definition below the type of a variable must be an
$m$-type and $\G$ denotes a context (the set of declarations of variables).

$$\hspace{.7cm} var\;\;\; \F{}{\G, x:A\;\v\;x:A}\vspace{.1cm}$$

\begin{tabular}{l l}
$\langle \: ,\rangle\;\;\; \F
{\G\;\v\;P_{1}:A_1\;\;\;\;\; \G\;\v\;P_{2}:A_2}{\G\;\v\;\lan P_{1},P_{2}\ran :A_{1}\wedge A_{2}}$ & $\;\;\;
\sigma_{i}\;\;\; \F {\G\;\v\;P_{i}:A_{i}}{\G\;\v\;{\sigma_{i}}\; (P_{i}):A_{1}\vee A_{2}}\;\;i\in\{1,2\}$\vspace{.4cm}\\

$\lambda \;\;\; \F {\G, x:A\;\v\;P:\bot}{\G\;\v\;\l
xP:A^{\bot}}$ & $\;\;\;\;
\star\;\;\; \F {\G\;\v\;P_{1}: A^{\bot}\;\;\;\;\;\G\;\v\;P_{2}:A}{\G\;\v\;(P_{1}\star P_{2}):\bot}$
\end{tabular}
\item We say that $M$ has type $A$, if there is a context $\Gamma$ such that $\Gamma\;\v\;M:A$. 
We consider $A$ as fixed for a certain element $\Gamma\;\v\;M:A$ of the typability relation.
\item As usual, we denote by $Fv(M)$, the set of the free variables of the term $M$.
\item The complexity of a term of ${\mathcal T}$ is defined as follows.
\begin{itemize}
\item[] $cxty(x)=0$,
\item[] $cxty(\lan  P_1 , P_2\ran)=   cxty((P_1\sta P_2)) =  cxty(P_1)+cxty(P_2)$,
\item[] $cxty(\l xP)=  cxty(\si_i(P)) = cxty(P)+1$, for $i\in \{ 1,2\}$.
\end{itemize}
\end{enumerate}
\end{defi}

\begin{defi}\hfill
\begin{enumerate}

\item The reduction rules are enumerated below.

\begin{tabular}{ l l l l l}
$(\b)$ & $(\l xP\star Q)$ & $\rightarrow_{\b}$ & $P[x:=Q]$ & $\;$ \\
$(\b^{\bot})$ & $(Q \star \l xP)$ & $\rightarrow_{\b^{\bot}}$ & $P[x:=Q]$ & $\;$ \\
$(\eta)$ & $\l x(P\star x)$ & $\rightarrow_{\eta}$ & $P$ & if $x\notin Fv(P)$\\
$(\eta^{\bot})$ & $\l x(x\star P)$ & $\rightarrow_{\eta^{\bot}}$ & $P$ & if $x\notin Fv(P)$\\
$(\pi)$ & $(\langle P_{1},P_{2}\rangle \star \si_{i}(Q_{i}))$ & $\rightarrow_{\pi}$ &  $(P_{i}\star Q_{i})$ & $i
 \in \{1,2\}$\\
$(\pi^{\bot})$ & $(\si_{i}(Q_{i})\star \langle P_{1},P_{2}\rangle)$ & $\rightarrow_{\pi^{\bot}}$ & $(Q_{i}\star
P_{i})$ & $i \in \{1,2\}$\\
$(Triv)$ & $E[P]$ & $\rightarrow_{Triv}$ & $P$ & $(*)$\\
\end{tabular}\medskip

\noindent{\rm (*)} If $E[-]$ is a context with type $\bot$ and $E[-]\neq[-]$, $P$ has type $\bot$ and $E[-]$ does not bind
any free variables in $P$.

\item Let us take the union of the above rules.
Let $\ra$ stand for the compatible closure of this union and, as usual, $\ra^*$ denote the reflexive, symmetric and transitive
closure of $\ra$. The notions of reduction sequence, normal form and normalization are defined with respect to $\ra$.

\item Let $M,N$ be terms. Assume $M \ra^* N$. The length (i.e. the number of steps) of the reduction $\ra^*$ is 
denoted by $lg(M \ra^* N)$.
\end{enumerate}

\end{defi}

\noindent We enumerate below some theoretical properties of the $\ls$-calculus following Berardi and Barbanera \cite{Ber-Bar} and de Groote \cite{de Gro2}.\newpage

\begin{prop}[Type-preservation property]
If $\G\;\v\; P:A$ and $P \ra^* Q$, then $\G \;\v\; Q:A$.
\end{prop}

\begin{prop}[Subformula property]
If $\Pi$ is a derivation of $\G \;\v\; P:A$ and $P$ is in normal
form, then every type occurring in $\Pi$ is a subformula of a type
occurring in $\G$, or a subformula of $A$.
\end{prop}

\begin{thm}[Strong normalization]\label{SN1}
If $\G \;\v\; P:A$, then $P$ is strongly normalizable, i.e. every reduction sequence starting from $P$ is finite.
\end{thm}

Berardi and Barbanera proved Theorem \ref{SN1} for the
extension of the $\ls$-calculus equivalent to first-order
Peano-arithmetic. The proof of this result by Berardi and Barbanera
\cite{Ber-Bar} is based on reducibility candidates, but the
definition of the interpretation of a type relies on
non-arithmetical fixed-point constructions.

We present a syntactical and arithmetical proof of the strong normalization of the $\ls$-calculus in Section 3.
The proof was inspired by a method of David and Nour \cite{Dav-Nou3}. First we establish that the permutation rules $\eta$, $\eta^\bot$ and
$Triv$ can be postponed so that we can restrict our attention uniquely to the rules $\b$, $\b^\bot$, $\pi$, and $\pi^\bot$.

\subsection{Permutation rules}

First of all, we prove that the $\eta$- and
$\eta^{\bot}$-reductions can be postponed w.r.t. $\b$, $\b^\bot$,
$\pi$, and $\pi^\bot$.

\begin{defi}\hfill
\begin{enumerate}

\item Let $\l_{\b\pi}$-calculus denote the calculus with only the reduction rules $\ra_\b$, $\ra_{\bb}$, $\ra_\pi$, and $\ra_{\pib}$.
\item Let $\ra_{\b \pi}$ stand for the union of $\ra_\b, \ra_{\bb}, \ra_\pi
, \ra_{\pib}$ and let $M\ra_{e}N$ denote the fact that
$M\ra_{\eta}N$ or $M\ra_{\eta^{\bot}}N$.
\item We denote by
$\ra_{\beta_{0}}$ (resp. by $\ra_{\beta^{\bot}_{0}}$) the $\b$-reduction $(\l
xM\sta N)\reb M[x:=N]$ (resp. the $\b^{\bot}$-reduction $(N\sta \l
xM)\rebb M[x:=N]$), where $x$ occurs at most once in $M$.
\item We use the standard notation $\ra^+$ and $\ra^*$ for the transitive and reflexive, transitive closure of a reduction,
respectively.
\end{enumerate}
\end{defi}

\noindent We examine the behaviour of a $\ree$ rule followed by a $\reb$ and a $\rebn$ rule in Lemmas \ref{ch3:ete} and \ref{ch3:bn}.

\begin{lem}\label{ch3:ete}
If $U \ree V \reb W$, then $U\reb V'\ra^*_e W$ or $U\rebn V'\reb W$ for some $V'$.
\end{lem}

\begin{proof}
We assume $\ree$ is an $\eta$-reduction, the proof of the case of $\reb$ is similar.
The proof is by induction on $cxty(U)$. The only interesting case is $U=(U_{1}\sta
U_{2})$. We consider only some of the subcases.
\begin{enumerate}

\item $U_{1}=\l x(U_{3}\sta x)$, with $x\notin Fv(U_3)$, and $V=(U_3\sta U_2)
\ra_{\b}U_4[y:=U_2]=W$, where $U_3=\l yU_4$. In this case $U=(\l x(U_{3}\sta
x)\sta U_{2})\rebn (U_{3}\sta U_{2})\reb U_{4}[y:=U_{2}]=W$, so
$\ra_{\eta} \reb$ turns into $\rebn \reb$.
\item $U_{1}=\l xU_{3}$, $U_{3}\ra_{\eta} U_{4}$ and $V=(\l xU_4\sta U_2)
\ra_{\b}U_{4}[x:=U_{2}]=W$. Then $U\reb V'=U_{3}[x:=U_{2}]\ra_{\eta}U_{4}[x:=U_{2}]=W$.
\item $U_{1}=\l xU_{3}$, $U_{2}\ra_{\eta} U_{4}$ and $V=(\l xU_3\sta U_4)
\ra_{\b}U_{3}[x:=U_{4}]=W$. Then $U\reb
V'=U_{3}[x:=U_{2}]\ra^*_{\eta}U_{3}[x:=U_{4}]$.\qedhere
\end{enumerate}
\end{proof}

\begin{lem}\label{ch3:bn}
If $U \ree V\rebn W$, then $U\rebn W$ or $U\rebn V'\ree W$ or $U\rebn V'\rebn W$ for some $V'$.
\end{lem}

\begin{proof}
By induction on $cxty(U)$. We assume $U=(U_1\sta U_2)$ and we consider some of the more interesting cases.
\begin{enumerate}

\item $U_{1}=\l x(U_{3}\sta x)$, with $x\notin Fv(U_3)$, and $V=(U_3\sta U_2)
\ra_{\b_0}U_4[y:=U_2]=W$, where $U_3=\l yU_4$. In this case $U=(\l x(U_{3}\sta
x)\sta U_{2})\rebn (U_{3}\sta U_{2})\rebn U_{4}[y:=U_{2}]=W$, thus
$\ra_{\eta} \rebn$ turns into $\rebn \rebn$.
\item $U_{1}=\l xU_{3}$, $U_{3}\ra_{\eta} U_{4}$ and $V=(\l xU_4\sta U_2)
\ra_{\b_0}U_{4}[x:=U_{2}]=W$. Then $U\rebn V'=U_{3}[x:=U_{2}]\ra_{\eta}U_{4}[x:=U_{2}]=W$.
\item $U_{1}=\l xU_{3}$, $U_{2}\ra_{\eta} U_{4}$ and $V=(\l xU_3\sta U_4)
\ra_{\b_0}U_{3}[x:=U_{4}]=W$. Then $U\rebn
V'=U_{3}[x:=U_{2}]\ra_{\eta}U_{3}[x:=U_{4}]$ provided $x$ occurs in $U_3$. Otherwise $U\rebn U_3=W$.\qedhere
\end{enumerate}
\end{proof}

\noindent We obtain easily the following lemma on the behaviour of several $\ree$ rules followed by a $\reb$ or a $\rebn$ rule.

\begin{lem}\label{ch3:bnk}
If $U{\ra^*_e} V\rebn W$, then $U{\rebn}^+ V'{\ra^*_e} W$ for some $V'$, and \\
$lg(U{\rebn}^+ V'{\ra^*_e} W)\leq lg(U{\ra^*_e} V\rebn W)$.
\end{lem}

\begin{proof}
By induction on $lg(U \ra^*_e V\rebn W)$, using Lemma \ref{ch3:bn}.
\end{proof}

\begin{lem}\label{ch3:incb}
If $U{\ra^*_e} V\reb W$, then $U{\reb}^+ V'\ra^*_e W$ for some $V'$.
\end{lem}

\begin{proof} By induction on $lg(U{\ra^*_e} V\reb W)$.
Use Lemmas \ref{ch3:ete} and \ref{ch3:bnk}.
\end{proof}

\begin{lem}\label{ch3:incbb}
If $U{\ra^*_e} V\rebb W$, then $U{\rebb}^+ V'{\ra^*_e} W$ for some $V'$.
\end{lem}

\begin{proof}
Similar to that of the previous lemma.
\end{proof}

We investigate now how a $\ree$ rule behaves when followed by a $\rep$ or $\ra_{\pib}$ rule.

\begin{lem}\label{ch3:etp}
If $U\ree V\rep W$ (resp. $U\ree V\ra_{\pib} W$), then  $U\rep V'\ree W$ or $U\rep W$
(resp. $U\ra_{\pib} V'\ree W$ or $U\ra_{\pib} W$) for some $V'$.
\end{lem}

\begin{proof}
Observe that in case of $U\ree V\rep W$ the following possibilities can occur:
either $U=\l x(V\sta x)$ and $V\rep W$ or $U=(\lan P_1,P_2\ran \sta \si(Q))$ and $V=(\lan P_1',P_2'\ran \sta \si(Q'))$,
 where exactly one of $P_i\ree P_i'$, $Q\ree Q'$ holds, the other two terms are left unchanged.
From this, the statement easily follows.
\end{proof}

\begin{lem}\label{ch3:inc}
If $U\ra^*_e V\ra_{\b \pi} W$, then $U\ra^{+}_{\b \pi} V'\ra_e W$ for some $V'$.
\end{lem}

\begin{proof}
By Lemmas \ref{ch3:incb}, \ref{ch3:incbb} and \ref{ch3:etp}.
\end{proof}

\begin{lem} \label{ch3:exch}
If $U\ra^*_e V\ra^*_{\b \pi} W$, then $U\ra^+_{\b \pi} V'\ra^*_e W$ for some $V'$.
\end{lem}

\begin{proof}
Follows from the previous lemma.
\end{proof}

We are now in a position to prove the main result of the section.

\begin{lem}\label{ch3:sne}
The $\eta$- and the $\etb$-reductions are strongly normalizing.
\end{lem}

\begin{proof}
The $\eta$- and $\etb$-reductions on $M$ reduce the complexity of $M$.
\end{proof}

\begin{defi}\hfill
\begin{enumerate}

\item Let $\lbe$-calculus denote the calculus obtained from the $\l_{\b\pi}$-calculus by adding the $\eta$- and
$\eta^{\bot}$-reductions to it.
\item Let $\ra_{\b \pi \eta}$ denote the union of $\ra_\b$,
$\ra_{\b^\bot}$, $\ra_\pi$, $\ra_{\pi^\bot}$, $\ra_\eta$ and
$\ra_{\eta^\bot}$.
\item Assume $M$ is a term strongly normalizable in the $\l_{\b \pi}$-calculus. Then we denote by $\eta_{\b \pi}(M)$
 the length of the longest reduction sequence $\ra^*_{\b\pi}$ starting from $M$.
\end{enumerate}
\end{defi}

\begin{cor}\label{ch3:b+e}
If the $\l_{\b\pi}$-calculus is strongly normalizing, then the $\lbe$-calculus is also strongly normalizing.
\end{cor}

\begin{proof}
 Let $M$ be a term, we prove by induction on $\eta_{\b \pi}
(M)$ that $M$ is strongly normalizable in the $\lbe$-calculus. Assume $S$ is an infinite ${\b \pi
\eta}$-reduction sequence starting from $M$. If $S$ begins with a $\ra_{\b \pi}$, then the induction
hypothesis applies. In the case when  $S$
contains only $\ra_e$-reductions, we are done by Lemma \ref{ch3:sne}. Otherwise
there is an initial subsequent $M\ra_e^+ M'\ra_{\b \pi}N$. By Lemma \ref{ch3:inc}, we have $M\ra^+_{\b
\pi}M''\ra^*_{e}N$. Thus, we can apply the induction hypothesis to
$M''$.
\end{proof}

In the rest of the section we deal with the rule $Triv$.
For strong normalization, it is enough to show
that $\ra_{Triv}$ can be postponed w.r.t. $\ra _{\b\pi\eta}$.

\begin{lem}\label{ch3:exct}
If $U\ra^*_{Triv} V\ra_{\b\pi\eta} W$, then $U\ra _{\b\pi\eta}^+ V'\ra^*_{Triv} W$ for some $V'$.
\end{lem}

\begin{proof}
It is enough to prove that if $U\ra_{Triv} V\ra_{\b\pi\eta} W$, then $U\ra _{\b\pi\eta} V'\ra_{Triv} W$ for some $V'$.
Observe that if $U=E[V]\ra_{Triv}V\ra_{\b\pi\eta} W$, then $V:\bot$ and $W:\bot$, from which the statement follows.
\end{proof}

\begin{lem}\label{ch3:tsn}
The reduction $\ra_{Triv}$ is strongly normalizing.
\end{lem}

\begin{proof}
The reduction $\ra_{Triv}$ on $M$ reduces the complexity of $M$.
\end{proof}

\begin{cor}\label{ch3:snlsfull}
If the $\l_{\b\pi}$-calculus is strongly normalizing, then the ${\lambda}^{\tiny{\textit{Sym}}}_{Prop}$-calculus
is also strongly normalizing.
\end{cor}

\begin{proof}
By Corollary \ref{ch3:b+e} and Lemmas \ref{ch3:exct} and \ref{ch3:tsn}.
\end{proof}

\section{Strong normalization of the $\l_{\b\pi}$-calculus}

In this section, we give an arithmetical proof for the strong normalization of the $\l_{\b\pi}$-calculus.
In the sequel we detail the proofs for
the $\b$- and $\pi$-reductions only, all the proofs below can be
extended with the cases of the $\b^{\bot}$- and $\pi^{\bot}$-reduction rules
in a straightforward way. We intend to examine how substitution behaves with respect to strong normalizability.
The first milestone in this way is Lemma \ref{ch3:nsnus}. Before stating the lemma, we formulate some auxiliary statements.

\begin{defi}\hfill
\begin{enumerate}

\item Let $SN_{\b\pi}$ denote the set of strongly normalizable terms of the $\l_{\b\pi}$-calculus.
\item Let $M\in SN_{\b\pi}$, then $\eta c(M)$ stands for the pair $\lan  \eta_{\b\pi}(M) , cxty(M)\ran$.
\end{enumerate}
\end{defi}

\begin{lem}\label{ch3:nsn} Let us suppose $M\in SN_{\b\pi}$, $N\in SN_{\b\pi}$ and
$(M\sta N)\notin SN_{\b\pi}$. Then there are $P\in SN_{\b\pi}$, $Q\in SN_{\b\pi}$
such that $M \ra^*_{\b\pi} P$ and $N \ra^*_{\b\pi}  Q$ and $(P\star Q)\notin SN_{\b\pi}$ is a redex.
\end{lem}

\begin{proof}
By induction on $\eta c(M)+\eta c(N)$. Assume $M\in SN_{\b\pi}$, $N\in SN_{\b\pi}$ and
$(M\sta N)\notin SN_{\b\pi}$. When $(M\sta N)\ra (M'\star N)$ or $(M\sta N)\ra (M\sta N')$, then the induction hypothesis applies.
Otherwise $(M\sta N)\ra P\notin SN_{\b\pi}$, and we have the result.
\end{proof}

\begin{defi}\hfill
\begin{enumerate}
\item A proper term is a term differing from a variable.
\item For a type $A$, $\Sigma_A$ denotes the set of simultaneous substitutions of the form
$[x_{1}:=N_{1},\ldots ,x_{k}:=N_{k}]$ where $N_i$ ($1 \leq i \leq n$) is proper and has type $A$.
\item A simultaneous substitution $\si \in \Sigma_A$ is said to be in $SN_{\b\pi}$, if, for every $x\in dom(\si )$, $\si
(x)\in SN_{\b\pi}$ holds.
\end{enumerate}
\end{defi}

\begin{lem} \label{ch3:ml}
Let $M,N$ be terms such that $M \ra^*_{\b\pi} N$.
\begin{enumerate}

\item If $N=\l xP$, then $M=\l xP_1$ with $P_1 \ra^*_{\b\pi} P$.
\item If $N=\lan P , Q\ran$, then $M=\lan P_1 , Q_1 \ran$ with $P_1 \ra^*_{\b\pi} P$ and $Q_1 \ra^*_{\b\pi} Q$.
\item If $N=\si_i (P)$, then $M=\si_i (P_1)$ with $P_1 \ra^*_{\b\pi} P$, for $i\in \{ 1,2\}$.
\end{enumerate}
\end{lem}

\begin{proof}
Straightforward.
\end{proof}

We remark that in the presence of the $\ra_{\eta}$ and $\ra_{\eta}^{\bot}$ rules the above lemma would not work.
For example, $\l x(y\star x)\ra_{\eta}y$.

\begin{lem}\label{ch3:mx}
If $M\in SN_{\b\pi}$ and $x \in Var$, then $(M\sta x)\in SN_{\b\pi}$ (resp. $(x\star M)\in SN_{\b\pi}$).
\end{lem}

\begin{proof}
Let us suppose $M\in SN_{\b\pi}$ and $(M\sta x)\notin SN_{\b\pi}$.
By Lemma \ref{ch3:nsn}, we must have $M \ra^*_{\b\pi} \l y M_1\in SN_{\b\pi}$
such that $(M\star x) \ra^*_{\b\pi} (\l yM_1\star x)\ra_{\b\pi}  M_1[y:=x]$ and $M_1[y:=x]\notin SN_{\b\pi}$.
Being a subterm of a reduct of $M\in SN_{\b\pi}$, we also have $M_1\in SN_{\b\pi}$. Moreover, $M_1[y:=x]$
is obtained from $M_1$ by $\a$-conversion, hence $M_1[y:=x]\in SN_{\b\pi}$, a contradiction.
\end{proof}

\begin{defi}\hfill
\begin{enumerate}

\item Let $M,N$ be terms.
\begin{enumerate}

\item We denote by  $M \leq N$ (resp. $M < N$) the fact that $M$ is a sub-term
(resp. a strict sub-term) of $N$.
\item We denote by $M \prec N$ the fact that $M \leq P$ for some $N \ra^+_{\b\pi} P$
   or $M < N$. We denote by $\preceq$ the
   reflexive closure of $\prec$.
\item Let $R$ be a ${\b\pi}$-redex. We write $M \ra^R N$ if $N$ is the term $M$
after the reduction of $R$.
\end{enumerate}
\item Let ${\mathcal R} = [R_1,\dots ,R_n]$ where $R_i$ is a ${\b\pi}$-redex $(1 \leq i \leq n)$.
Then ${\mathcal R}$ is called zoom-in if, for every $1\leq i< n$,
$R_i\ra^{R_i} R_i'$ and $R_{i+1} \leq R_i'$. Moreover, ${\mathcal R}$ is minimal, if, for each $R_i=(P_i\star Q_i)$,
we have $P_i$, $Q_i\in SN_{\b\pi}$ and $(P_i\star Q_i)\notin SN_{\b\pi}$. We write $M \ra^{\mathcal R} N$, if
$M \ra^{R_1} ... \ra^{R_n} N$.
\end{enumerate}
\end{defi}

\noindent For the purpose of proving the strong normalization of the calculus,
it is enough to show that the set of strongly normalizable terms is closed under substitution.
To this end, we show that, if $U$, $S\in SN_{\b\pi}$ and $U[x:=S]\notin SN_{\b\pi}$, then there is a term $W\leq U$ of a
special form such that $W\in SN_{\b\pi}$ and $W[x:=S]\notin SN_{\b\pi}$. Moreover, we show that the sequence of redexes
leading to $W$ is not completely general, this is a zoom-in sequence defined below. Reducing the outermost redexes of a
zoom-in sequence preserve some useful properties,
which is the statement of Lemma \ref{ch3:zoom}.

\begin{lem} \label{ch3:nsnus}
Let $U$, $S\in SN_{\b\pi}$ and suppose $U[x:=S]\notin
SN_{\b\pi}$. Then there are terms $P,V\preceq U$ and a zoom-in minimal ${\mathcal R}$ such that
$U[x:=S] \ra^{\mathcal R} V[x:=S]$, $(x\star P) \leq V$ (or $(P\star x) \leq V$), $P[x:=S]\in SN_{\b\pi}$
and $(x\star P)[x:=S]\notin SN_{\b\pi}$ (or $(P\star x)[x:=S]\notin SN_{\b\pi}$).
\end{lem}

\begin{proof}
The proof goes by induction on $\eta c(U)$. If $U$ is other than an application, we can apply the induction hypothesis.
Assume $U=(U_1\star U_2)$ with $U_i[x:=S]\in SN_{\b\pi}$ $(i\in \{1,2\})$ and $U[x:=S]\notin SN_{\b\pi}$.
By Lemma \ref{ch3:nsn} and the induction hypothesis we may assume that $(U_1\star U_2)[x:=S]\ra_{\rho} U'\notin SN_{\b\pi}$,
where $\rho\in \{\beta,\beta^\bot,\pi,\pi^\bot\}$. Let us suppose $\rho=\beta$, the other cases can be treated similarly.
If $U_1=\l yU_1'$, then the induction hypothesis applies to $U_1'[y:=U_2]$. Otherwise $U_1=x$, and we have obtained the result.
\end{proof}

Next we define an alternating substitution: we start from two sets of terms of complementary types and the substitution is
defined in a way that we keep track of which newly added sets of substitutions come from which of the two sets.
The reason for this is that Lemma \ref{ch3:nsnus} in itself is not enough for proving the strong normalizability
of \smash{$\ls$} even if we would consider the $\beta$ and $\beta^\bot$ rules alone.
We have to show that, if we start from a term $(U_1\star U_2)$, where $U_1$ and $U_2\in SN_{\b\pi}$ and we assume that
$U_1[x:=U_2]\notin SN_{\b\pi}$, then there are no deep interactions between the terms which come from $U_1$ and from $U_2$.
We can identify a subterm of a reduct of $U_1$ which is the cause for being non $SN_{\b\pi}$, when performing a substitution
with $U_2$.

\begin{defi}\label{ch3:lssimsub}\hfill
\begin{enumerate}

\item A set $\ma{A}$ of proper terms is called $\preceq$-closed from
below if, for all terms $U,U'$, if $U'\preceq U$, $U\in \ma{A}$ and $U'$
is proper, then $U'\in \ma{A}$.
\item Let $\ma{A},\ma{B}$ be sets $\preceq$-closed from below and $A$ a type.
We define simultaneously two sets of substitutions
\begin{enumerate}

\item $\Pi_{A}(\ma{B})\subseteq
\S_{A}$ and $\Theta_{A^{\bot}}(\ma{A})\subseteq \S_{A^{\bot}}$ as follows.
\begin{itemize}
\item[]
\begin{itemize}
\item
$\emptyset \in \Pi_{A}(\ma{B})$,
\item $[y_1:=V_1\t_1,\dots ,y_m:=V_m\t_m]\in \Pi_A(\ma{B})$ if $V_i\in
\ma{B}$ such that $type(V_i)=A$ and $\t_i\in \Theta_{A^{\bot}}(\ma{A})$ $(1\leq i\leq m)$.
\end{itemize}
\item[]
\begin{itemize}
\item $\emptyset \in\Theta_{A^{\bot}}(\ma{A})$.
\item $[x_1:=U_1\r_1,\dots ,x_m:=U_m\r_m]\in\Theta_{A^{\bot}}(\ma{A})$
if $U_i\in \ma{A}$ such that $type(U_i)=A^{\bot}$ and $\r_i\in \Pi_A(\ma{B})$
$(1\leq i\leq m)$.
\end{itemize}
\end{itemize}
\item Let $\ma{S}_{A}(\ma{A},\ma{B}) =\{U\r \;|\; U\in \ma{A}
\textrm{ and } \r \in \Pi_A(\ma{B})\}\cup \{V\t \;|\; V\in \ma{B}
\textrm{ and } \t \in\Theta_{A^{\bot}}(\ma{A})\}$.
It is easy to see that, from $U\leq V$ and $V\in \ma{S}_{A}(\ma{A},\ma{B})$, it follows $U\in \ma{S}_{A}(\ma{A},\ma{B})$.
\end{enumerate}
\end{enumerate}
\end{defi}

\begin{lem}\label{ch3:zoom} Let $n$ be an integer, $A$ a type
of length $n$ and  ${\mathcal R} =[R_1,\dots ,R_m]$
a zoom-in minimal sequence of redexes. Assume the property $H$  ``if $U$,
$V\in SN_{\b\pi}$ and $cxty(type(V))<n$, then $U[x:=V]\in SN_{\b\pi}$'' holds.  If $R_1\in\ma{S}_{A}(\ma{A},\ma{B})$ for some
sets $\ma{A}$ and $\ma{B}$ $\preceq$-closed from below, then
$R_m\in \ma{S}_{A}(\ma{A},\ma{B})$.
\end{lem}

\begin{proof} The proof goes by induction on $m$. We prove the induction step from $m=1$ to $m=2$, the proof is
the same when $m\in \mathbb{N}$ is arbitrary. We only
treat the more interesting cases. Assume $R_1\in \ma{S}_{A}(\ma{A},\ma{B})$.\enlargethispage{\baselineskip}
\begin{enumerate}

\item $R_1=(\l xQ\star S) \ra_{\b} R_1'=Q[x:=S]$ and $R_2\leq R_1'$.
\begin{enumerate}

\item Suppose $R_1=U\rho$ for some $U\in
\ma{A}$ and $\rho \in \Pi_A(\ma{B})$. Then $U=(U_1\star U_2)$
with $U_1\rho =\l xQ$ and $U_2\rho = S$, and, since $\rho\in \S_{A}$, $U_1$ must be proper.
Then we have $U_1=\l xU_1'$ and
$U_1'\rho = Q$ for some $U_1'$. Now,
$R_1'=U_1'[x:=U_2]\rho\in \ma{S}_{A}(\ma{A},\ma{B})$, which yields $R_2\in \ma{S}_A(\ma{A},\ma{B})$.
\item Assume now $R_1=V\tau$. Then $V=(V_1\star V_2)$
with $V_1\tau =\l xQ$ and $V_2\tau = S$, and, since $\tau\in \S_{A^\bot}$, $V_2$ must be proper.
If $V_1$ is proper, then, as before, we obtain the result. Otherwise $V_1\tau =U\rho =\l xQ$.
Since $U\in \ma{A}$ is proper, $U=\l xU_1$ and
$U_1\rho =Q$ for some $U_1$. Then $U_1\rho_1\in \ma{S}_{A}(\ma{A},\ma{B})$ with $\rho_1=\rho + [x:=V_2\tau]$,
since $type(V_2\tau)=type(S)=A$. This implies $R_2\in \ma{S}_{A}(\ma{A},\ma{B})$.
\end{enumerate}
\item $R_1=(\lan Q_1,Q_2\ran\star \sigma_1(S))\ra_{\pi}(Q_1\star S) =R_1'$ and $R_2\leq
R_1'$.
\begin{enumerate}

\item
Assume $R_1=U\rho$ for some $U\in \ma{A}$ and $\rho \in \Pi_A(\ma{B})$. Then $U_1\rho=\lan Q_1,Q_2\ran$ and
$U_2\rho=\sigma_1(S)$.
\begin{itemize}
\item[-] Let $U_1$ and $U_2$ be proper. Then $U_1=\lan U_1',U_1''\ran$ and $U_2=\sigma_1(U_2')$
such that $U_1'\rho=Q_1$, $U_1''\rho=Q_2$ and $U_2'\rho=S$. We have $R_1'=(U_1'\star U_2')\rho\in\ma{S}_{A}(\ma{A},\ma{B})$,
which yields the result.
\item[-] Assume $U_2\in Var$. Then $V\tau =\sigma (S)$, and $cxty(type(S))<$

$cxty(type(\sigma S))=n$.
Then assumption $H$ and the fact that $\lan Q_1,Q_2\ran\in SN_{\b\pi}$, together with Lemma \ref{ch3:mx},
lead to $(Q_1\star S)\in SN_{\b\pi}$, which is not possible.
Since $\rho\in \Sigma_A$, $U_1\in Var$ is impossible.
\end{itemize}
\item Assume $R_1=V\tau$ for some $V\in \ma{B}$ and $\tau \in \Theta_{A^{\bot}}(\ma{A})$. Then $V_1\tau=\lan Q_1,Q_2\ran$ and
$V_2\tau=\sigma_1(S)$, where $V=(V_1\star V_2)$.
\begin{itemize}
\item[-] Let $V_1$ and $V_2$ be proper. Then $V_1=\lan V_1',V_1''\ran$ and $V_2=\sigma_1(V_2')$
such that $V_1'\tau=Q_1$, $V_1''\tau=Q_2$ and $V_2'\tau=S$. We have $R_1'=(V_1'\star V_2')\tau\in\ma{S}_{A}(\ma{A},\ma{B})$.
\item[-] Assume $V_1\in Var$. Then $U\rho =\lan Q_1,Q_2\ran$, where $cxty(type(Q_1))<$

$cxty(type(\lan Q_1,Q_2\ran))=n$.
Then assumption $H$ and the fact that $S\in SN_{\b\pi}$, together with Lemma \ref{ch3:mx}, lead to $(Q_1\star S)\in SN_{\b\pi}$,
which is not possible. Since $\tau\in\Sigma_{A^{\bot}}$, the case of $V_2\in Var$ is impossible.\qedhere
\end{itemize}
\end{enumerate}
\end{enumerate}
\end{proof}

\noindent The next lemma identifies the subterm of $U$ being responsible for the non strong normalizability of $U[x:=V]$.

\begin{lem} \label{ch3:lsmain}
Let $n$ be an integer and $A$ a type of length $n$. Assume the property $H$ ``if $U$,
$V\in SN_{\b\pi}$ and $cxty(type(V))<n$, then $U[x:=V]\in SN_{\b\pi}$'' holds.
\begin{enumerate}

\item Let $U$ be a proper term, $\si
\in \S_{A}$ and $a\notin Im(\si)$. If $U\si, P\in SN_{\b\pi}$ and $U\si [a:=P]\notin SN_{\b\pi}$, then
there exists $U'$ such that $(U'\star a)\preceq U$ and
$\si'\in \S_{A}$ such that $U'\si'\in SN_{\b\pi}$ and $(U'\si'\star a)[a:=P] \notin SN_{\b\pi}$.
\item Let $U$ be a proper term, $\si
\in \S_{A^{\bot}}$ and $a\notin Im(\si)$. If $U\si, P\in SN_{\b\pi}$ and $U\si [a:=P]\notin SN_{\b\pi}$, then
there exists  $U'$ such that $(a\star U')\preceq U$ and
$\si'\in \S_{A^\bot}$ such that $U'\si'\in SN_{\b\pi}$ and $(a\star U'\si')[a:=P] \notin SN_{\b\pi}$.
\end{enumerate}
\end{lem}

\begin{proof}
Let us consider only case (1). We identify the reason of $U\si [a:=P]$ being non strongly normalizable,
we find a subterm $(U'\star a)$ of a reduct of $U$ such that, for a substituted instance of $(U'\star a)$,
$(U'\star a)\si'\in SN_{\b\pi}$ and $(U'\star a)\si'[a:=P]\notin SN_{\b\pi}$. This will contradict by
some minimality assumption concerning $U$ in the next lemma. For this we define two sets
 of substitutions as in Definition \ref{ch3:lssimsub} with the sets $\mathcal{A}$ and
$\mathcal{B}$ as below. We note that Property H of the previous lemma implicitly ensures
that the type of $U$ and the type of the elements in $\si$ can be of the same lengths.\enlargethispage{\baselineskip}

Let
$$\mathcal{A}=\{M\,|\,M\preceq U \textrm{ and } M \textrm{ is proper} \},$$
$$\mathcal{B}=\{V\,|\,V\preceq \si (b)\textrm{ for some } b \in dom(\si )
\textrm{ and } V \textrm{ is proper} \}.$$
Then $U\si \in \ma{S}_{A}(\ma{A},\ma{B})$. By Lemma \ref{ch3:nsnus}, there exists a minimal zoom-in
${\mathcal R}=[R_1,\dots,R_n]$ and there are terms
$U^*$ and $V\preceq U\si$ such that $U\si [a:=P]\ra^{\mathcal R} V[a:=P]$ and $(U^*\star a)\leq V$ and
$(U^*\star a)\in SN_{\b\pi}$ and $(U^*\star a)[a:=P]\notin SN_{\b\pi}$ or $(a\star U^*)\leq V$ and $(a\star U^*)\in SN_{\b\pi}$
and $(a\star U^*)[a:=P]\notin SN_{\b\pi}$. Assume the former. By Lemma \ref{ch3:zoom},
$(U^*\star a)\in \ma{S}_{A}(\ma{A},\ma{B})$. Then $(U^*\star a)=S\rho$ for some $S \in \ma{A}$ or $(U^*\star a)=W\tau$
for some $W\in \ma{B}$. Since $a\notin Im(\si )$, the latter is impossible. The former case, however, yields $S=(U'\star a)$
with $U'\rho=U^*$ for some $U'\in \ma{A}$, which proves our assertion.
\end{proof}

The next lemma states closure of strong normalizability under substitution.

\begin{lem}\label{ch3:sn}
If $M,N\in SN_{\b\pi}$, then $M[x:=N]\in SN_{\b\pi}$.
\end{lem}

\begin{proof} We are going to prove a bit more general statement. Suppose
$M, N_{i}\in SN_{\b\pi}$ are proper, $type(N_{i})=A$ $(1\leq i\leq k)$.
Let $\t _{i}\in
\Sigma_{A^{\bot}}$ are such that $\t _{i}\in
SN_{\b\pi}$ $(1\leq i\leq k)$ and let $\r =[x_{1} :=N_{1}\t _{1},\ldots
,x_{k} :=N_{k}\t _{k}]$. Then we have $M \r \in SN_{\b\pi}$.
The proof is by induction on $(cxty(A), \eta_{\b\pi}(M), cxty(M)$, $\Sigma_i \;
\eta_{\b\pi}(N_i), \Sigma_i \; cxty(N_i))$ where, in $\Sigma_i \; \eta_{\b\pi}(N_i)$
and $\Sigma_i \; cxty(N_i)$, we count each occurrence of the
substituted variable. For example if $k=1$ and $x_1$ has $n$
occurrences, then $\Sigma_i \; \eta_{\b\pi}(N_i)=n\cdot \eta_{\b\pi}(N_1)$.

The only nontrivial case is when $M=(M_{1}\sta M_{2})$ and
$M\r\notin SN_{\b\pi}$. By the induction hypothesis $M_{i}\r\in SN_{\b\pi}$ $(i\in
\{1,2\})$. We select some of the typical cases.

\begin{enumerate}

\item[(A)] $M_{1}\r \ra_{\b\pi} \l zM'$ and $M'[z:=M_2]\notin SN_{\b\pi}$.

\begin{enumerate}

\item[1.] $M_{1}$ is
proper, then there is an $M_{3}$ such that $M_{1}=\l zM_{3}$ and
$M_{3}\r \ra_{\b\pi} M'$. In this case $(M_{3}[z:=M_{2}])\r \notin SN_{\b\pi}$
and since $\eta_{\b\pi}(M_{3}[z:=M_{2}])<\eta_{\b\pi}(M)$, the induction
hypothesis gives the result.

\item[2.] $M_{1}\in Var$. Then $M_{1}=x\in dom(\r)$,
$\r(x)=N_{j}\t_{j}\ra_{\b\pi} \l zM'$ for some $(1\leq j\leq k)$. Since
$N_{j}$ is proper, there is an $N'$ such that $N_{j}=\l zN'$,
$N'\t_{j}\ra_{\b\pi} M'$. Then $N'\t_{j}[z:=M_{2}\r]\notin SN_{\b\pi}$ and
$type(z)={type({N_{j}})}^{\bot}=type(\t_{j})$, so, by the previous
lemma, we have $N''\prec N'$ and $\t '$ such that $(N''\t
'\sta M_{2}\r )\notin SN_{\b\pi}$.  Now we have $(N''\t '\sta M_{2}\r
)=(y\sta M_{2}\r )[y:=N''\t ']$, $type(N'')={type(\t ')}^{\bot}=A$
and $\eta_{\b\pi} cxty(N'')<\eta c(N_{j})$, which contradicts the induction
hypothesis.

\end{enumerate}

\item[(B)] $M_{1}\r \ra_{\b\pi} \lan M',M''\ran  $ and either
$(M'\sta M_2)\notin SN_{\b\pi}$ or $(M''\sta
M_2)\notin SN_{\b\pi}$. Suppose the former.

\begin{enumerate} 

\item[1.] $M_{1},M_{2}$ are proper, then there are
$M_{3},M_{4}$ such that $M_{1}=\lan M_{3},M_{4}\ran$
and $M_{3}\r \ra_{\b\pi} M'$, or $M_{4}\r \ra_{\b\pi} M''$. Assume the former.
Then we have $(M_{3}\sta M_{2})\r \notin SN_{\b\pi}$ and
$\eta_{\b\pi}((M_{3}\sta M_{2}))<\eta_{\b\pi}(M)$, a contradiction.

\item[2.] $M_{1}=x\in dom(\r)$, then $\r (x)=N_{j}\t_{j}\ra_{\b\pi} \lan M',M''\ran$,
$N_{j}$ is proper and $N_{j}=\lan U,V\ran
$, $U\t_{j}\ra_{\b\pi} M'$ or $V\t_{j}\ra_{\b\pi} M''$. Now $(U\t_{j}\sta
M_2)=(y\sta M_2)[y:=U\t_{j}]\notin SN_{\b\pi}$, but
$cxty(type(U))<cxty(type(N_{j}))$, a contradiction again.

\item[3.] $M_{2}\in Var$. This is similar to the previous case. By the
same argument as in part (A)-2.-(a) of the proof of the previous
lemma, $M_{1}$ and $M_{2}$ cannot be both variables. This completes
the proof of the lemma.\qedhere
\end{enumerate}
\end{enumerate}
\end{proof}

\begin{thm} \label{ch3:snt}
The $\l_{\b\pi}$-calculus is strongly normalizing.
\end{thm}

\begin{proof} It is enough to show that,
for every term, $M$, $N\in SN_{\b\pi}$ implies $(M\sta N)\in SN_{\b\pi}$.
Supposing $M,N\in SN_{\b\pi}$, Lemma \ref{ch3:mx} gives $(M\sta x)\in SN_{\b\pi}$,
which yields, by the previous lemma, $(M\sta N)=(M\sta x)[x:=N]\in
SN_{\b\pi}$.\end{proof}
\newpage

\section{The $\lmt$- and the $\lmts$-calculus}\label{section3}

In this section, we introduce the $\lmt$-calculus together with one of its extensions,
the $\lmts$-calculus, by which we establish a translation of the $\ls$-calculus and thus obtain
the strong normalization of the $\lmts$-calculus as a consequence.

\subsection{The $\lmt$-calculus}

The $\lmt$-calculus was introduced by Curien and Herbelin
(\cite{Her} and \cite{Cur-Her}). We examine here the calculus defined
by Curien et al. \cite{Cur-Her}, which is a simply typed one. The
$\lmt$-calculus was invented for representing proofs in classical
Gentzen-style sequent calculus: under the Curry-Howard
correspondence a version of Gentzen-style sequent calculus is
obtained as a system of simple types for the $\lmt$-calculus.
Moreover, the system presents a clear duality between
call-by-value and call-by-name evaluations.

\begin{defi}
There are three kinds of terms, defined by the following grammar,
and there are two kinds of variables.
We assume that we use the same set of variables in the $\lmts$-calculus, too.
In the literature, different authors use different terminology. Here, we will call them either $c$-terms, or $l$-terms
or $r$-terms. Similarly, the variables will be called either
$l$-variables (and denoted as $x,y,...$) or  $r$-variables (and denoted as $a, b, ...$).

\[
\begin{array}{ccccccccc}
p   &::=    & \lfl t , e \rfl & \, &\, &\, &\, &\, &\, \\
t   &::=    & x &\mid & \l x t &\mid &\mu \a  p &\, &\, \\
e   &::=    &\a &\mid & (t.e)  &\mid & \tilde{\mu}x  p &\, &\,
\end{array}
\]
As usual, we denote by $Fv(u)$, the set of the free variables of the term $u$.
\end{defi}

\begin{defi}
The types are built from atomic formulas
(or, in other words, atomic types) with the connector
$\rightarrow$.
We assume that the same set of type variables ${\mathcal A}$ is used in the $\lmts$-calculus, also.
The typing system is a sequent calculus based on
judgements of the following form.
\begin{center}
$p : (\G \;\vv\; \D )$   $\;\;\;\;\;\;\;\;\;$ $\G \;\vv\; t : A\;
|\;
\D$ $\;\;\;\;\;\;\;\;\;$ $\G\; |\; e : A \;\vv\; \D$
\end{center}
where $\G$ (resp. $\D$) is a set of declarations of the form $x :
A$ (resp. $a : A$), $x$ (resp. $a$) denoting a $l$-variable (resp.
an $r$-variable) and $A$ representing a type, such that $x$ (resp.
$a$) occurs at most once in an expression of $\G$ (resp. $\D$) of the
form $x:A$ (resp. $a: A$). We say that $\G$ an $l$-context and $\D$ is an $r$-context, respectively.
The typing rules are as follows
\[
\begin{array}{ll}

\begin{minipage}[t]{180pt}
$Var_1 \;\;\;\F{}{\G, x : A \;\vv\; x : A\; |\; \D }$\\
\end{minipage}
&
\begin{minipage}[t]{180pt}
$Var_2 \;\;\;\F{}{\G  \; |\; \a : A\;\vv\; \a : A , \D }$\\[0.3cm]
\end{minipage}
\\

\begin{minipage}[t]{180pt}
$\l \;\;\;\F{\G, x : A \;\vv\; t : B\; |\;  \D}{\G\;\vv\; \l x t : A \ra B\; |\;  \D}$\\
\end{minipage}
&
\begin{minipage}[t]{180pt}
$(.) \;\;\;\F{\G \;\vv\; t : A\; |\;  \D \;\;\; \G \; |\;   e : B \;\vv\;
    \D}{\G  \; |\; (t.e) : A \ra B \;\vv\; \D}$\\[0.1cm]
\end{minipage}
\end{array}
\]

$$\lfl ,\rfl \;\;\;\F{\G \;\vv\; t : A \; |\;  \D \;\;\; \G \; |\; e : A\;\vv\;
\D}{\lfl t , e \rfl : (\G \;\vv\; \D )}$$

\[
\begin{array}{ll}
\begin{minipage}[t]{180pt}
$\mu \;\;\;\F{p : (\G \;\vv\; \a : A , \D)}{\G \;\vv\; \mu \a p : A\; |\;  \D}$\\
\end{minipage}
&
\begin{minipage}[t]{180pt}
$\tilde{\mu} \;\;\;\F{p : (\G , x : A \;\vv\;  \D)}{\G \; |\; \tilde{\mu} x p : A\;\vv\; \D}$\\
\end{minipage}
\end{array}
\]
\end{defi}
\newpage

\begin{defi}
The cut-elimination procedure (on the logical side) corresponds to
the reduction rules (on the terms) given below.\\

\begin{tabular}{ l l l l l}
$(\l)$ & $\lfl \l x  t , (t'.e) \rfl$ & $\raa_{\;\l}$ & $\lfl t', \tilde{\mu} x \, \lfl t , e \rfl \rfl$ \\
$(\mu)$ & $\lfl \mu \a p , e \rfl$ & $\raa_{\;\mu}$ & $p[\a := e]$ & $\;$ \\
$(\tilde{\mu})$ & $\lfl  t , \tilde{\mu} x p \rfl$ & $\raa_{\;\tilde{\mu}}$ & $p[x:= t]$ & \\
$(s_l)$ & $\mu \a \lfl t , \a \rfl$ & $\raa_{\;s_l}$ & $t$ & ${\rm if} \; \a \not \in Fv(t)$\\
$(s_r)$ & $\tilde{\mu} x \lfl  x ,  e \rfl$ & $\raa_{\;s_r}$ &  $e$ & ${\rm if} \; x \not \in Fv(e)$\\
\end{tabular}
\item Let us take the union of the above rules.
Let $\raa$ stand for the compatible closure of this union and, as usual, $\raa^*$ denote the reflexive, symmetric and transitive
closure of $\raa$. The notions of reduction sequence, normal form and normalization are defined with respect to $\raa$.
\end{defi}

We present below some theoretical properties of the $\lmt$-calculus
(Herbelin \cite{Her}, Curien and Herbelin \cite{Cur-Her}, de Groote \cite{de Gro2}, Polonovski \cite{Poi} and David and Nour \cite{Dav-Nou4}).

\begin{prop}[Type-preservation property]\label{ch4:pres}
If $\G\;\vv\; t:A\;|\;\D$ (resp.  $\G\;|\;e:A\;\vv\;\D$, resp.
$p:(\G\;\vv\;\D)$) and $t \raa^* t'$ (resp. $e\raa^* e'$, resp. $p \raa^*
p'$), then $\G\;\vv\; t':A\;|\;\D$ (resp. $\G\;|\;e':A\;\vv\;\D$, resp.
$p':(\G\;\vv\;\D)$).
\end{prop}

\begin{prop}[Subformula property]\label{ch4:subform}
If $\Pi$ is a derivation of $\G\;\vv\; t:A\;|\;\D$ (resp.  $\G\;|\;e:A\;\vv\;\D$, resp.
$p:(\G\;\v\;\D)$) and $t$ (resp. $e$, resp. $p$) is in normal
form, then every type occurring in $\Pi$ is a subformula of a type
occurring in $\G \cup \D$, or a subformula of $A$ (only for $t$ and $e$).
\end{prop}

\begin{thm}[Strong normalization property]\label{SN2}
If $\G\;\vv\; t:A\;|\;\D$ (resp.  $\G\;|\;e:A\;\vv\;\D$, resp.
$p:(\G\;\vv\;\D)$), then $t$ (resp. $e$, resp. $p$) is strongly normalizable,
i.e. every reduction sequence starting from $t$
(resp. $e$, resp. $p$) is finite.
\end{thm}

The proof of Theorem \ref{SN2} can be found in the thesis of Polonovski \cite{Poi},
as well as in the work of David and Nour \cite{Dav-Nou4}, where an
arithmetical proof is presented.

\subsection{The $\lmts$-calculus}

Since we work in a sequent calculus, where negation is implicitly built in the rules,
the typing rules of the $\lmt$-calculus do not handle negation.
However, for a full treatment of propositional logic we
found it more convenient to introduce rules concerning negation. Since $c$-terms, which could
have been candidates for objects of type $\bot$, are distinctly
separated from terms, adding new term- and type-forming operators seems to be the easiest way to
define negation.

\begin{defi}\label{ch4:defterms}\hfill
\begin{enumerate}

\item The terms of the $\lmts$-calculus
are defined by the following grammar.
\[
\begin{array}{ccccccccc}
p   &::= & \lfl t , e \rfl & \, &\, &\, &\, &\, &\, \\
t &::= & x &\mid & \l x \, t &\mid &\mu \a \, p &\mid &\overline{e} \\
e &::= &\a &\mid & (t.e)  &\mid & \tilde{\mu}x \, p &\mid
&\widetilde{t}
\end{array}
\]
As an abuse of terminology, in the sequel when speaking about the
syntactic elements of the $\lmts$-calculus, we may not distinguish
$l$-, $r$- and $c$-terms, we may speak about terms in general.
We denote by $\mathfrak{T}$ the set of terms of the $\lmts$-calculus.
\newpage

\item The complexity of a term of $\mathfrak{T}$ is defined as follows.
\begin{itemize}
\item[] $cxty(x)=cxty(\a)=0$,
\item[] $cxty(\l x t)=   cxty(\widetilde{t}) =  cxty(t)+1$,
\item[] $cxty(\overline{e}) = cxty(e) +1$,
\item[] $cxty(\mu \a p)=   cxty(\tilde{\mu}x \, p) =  cxty(p)+1$,
\item[] $cxty(\lfl t , e \rfl )=   cxty((t.e) ) =  cxty(t)+cxty(e)$.
\end{itemize}
\end{enumerate}
\end{defi}

\begin{defi}
The type inference rules are the same as in the $\lmt$-calculus with two
extra rules added for the types of the complemented terms.
Moreover, we introduce an equation between types (for all types $A$, $(A^\bot)^\bot = A$)
to ensure that our negation is involutive.
\[
\begin{array}{ll}
\begin{minipage}[t]{180pt}
$\;\;\;\;\;\;\;\;\;\;\;\;\;\;\;\;\;\;\;\;\;\;\overline{.}\;\;\;\F{\G\; |\; e : A \vv \D}{\G \vv \overline{e}:A^\bot\; |\; \D }$\\
\end{minipage} &
\begin{minipage}[t]{180pt}
$\widetilde{.}\;\;\; \F{\G \vv t : A\; |\;  \D }{\G  \; |\; \widetilde{t}:A^\bot\vv \D }$\\
\end{minipage}
\end{array}
\]
We also define the complexity of types in the $\lmts$-calculus.

\begin{enumerate}

\item[] $cxty(A)=0$ for atomic types,
\item[] $cxty(A\ra B)=cxty(A)+cxty(B)+1$,
\item[] $cxty(A^{\bot})=cxty(A)$.
\end{enumerate}
That is, the complexity of a type $A$ provides us with the number of arrows in $A$. The presence of negation makes it
necessary for us to introduce new rules handling negation.
\end{defi}

\begin{defi}
Besides the reduction rules already present in $\lmt$, we endow the
calculus with some more new rules to handle the larger set of
terms. In what follows $cl$ stands for the name: complementer
rule. We shall refer to the $cl_{1,l}$- and $cl_{1,r}$-rules by a 
common notation as the $cl_{1}$-rules.

\begin{tabular}{ l l l l l}
$(cl_{1,l})$ & $\overline{\widetilde{t}}$ & $\raa_{\; cl_{1,l}}$ & $t$ \\
$(cl_{1,r})$ & $\widetilde{\overline{e}}$ & $\raa_{\; cl_{1,r}}$ & $e$ & $\;$ \\
$(cl_{2})$ & $\lfl \overline{e} , \widetilde{t}\rfl$ & $\raa_{\; cl_{2}}$ & $\lfl t , e\rfl$ & \\
\end{tabular}\medskip

\noindent In the sequel, we continue to apply the notation $\raa$ and $\raa^*$ in relation with this new calculus.
\end{defi}

Obviously, the statements analogous to Propositions \ref{ch4:pres} and \ref{ch4:subform} are still valid.

\section{Relating the \texorpdfstring{$\ls$}{lambda-Sum-Prop}-calculus to the \texorpdfstring{$\lmts$}{blambda-mu-bmu}-calculus}\label{section4}

Rocheteau \cite{Roc} defined a translation between the $\lmt$-calculus and the $\lm$-calculus, treating both a call-by-value
and a call-by-name aspect of $\lmt$. In this subsection, we give a translation (in both directions) between the \smash{$\ls$}-calculus
and the $\lmts$-calculus, which is a version of the $\lmt$-calculus extended with negation. The translations are such that strong
normalization of one calculus follows from that of the other in both directions. We omit issues of evaluation strategies, however.
In the end of the section we give an exact description of the correspondence between the two translations. Preparatory to presenting
the translations, let us introduce some definitions and notation below. We assume that the two calculi have the same sets of variables
and atomic types. Moreover, as an abuse of notation, if $\overline{{\a}}:A^{\bot}$ stems from the $r$-variable $\a:A$ in the
$\lmts$-calculus, then we suppose that in the \smash{$\ls$}-calculus $\overline{{\a}}$ denotes a variable with type $A^{\bot}$.

\subsection{A translation of the
  \texorpdfstring{$\lmts$}{blambda-my-bmu}-calculus into the
  \texorpdfstring{$\ls$}{lambda-Sym-Prop}-calculus}

\begin{defi}\label{ch5:pi}\hfill
\begin{enumerate}

\item Let us consider the $\ls$-calculus. For $i\in\{ 1 , 2\}$, we write ${\pi}_i(y)=\l z(y\star {\si}_i(z))$.
Then, we can observe that $y : A_1 \wedge A_2 \v {\pi}_i(y):A_i$, for $i\in\{ 1 , 2\}$.

\item We define a translation $.^\mathfrak{e}:\mathfrak{T} \longrightarrow \ma{T}$ as follows.
\[ p^\mathfrak{e}=(u^\mathfrak{e}\star v^\mathfrak{e}) \;\;\;
\textrm{ if $\;\;p=\lfl v , u\rfl$}.
\]
\[ t^\mathfrak{e}=\left\{ \begin{array}{ll} x & \;\;\;\textrm{ if $\;\;t=x$},\\
\l y(\l x({\pi}_2(y)\star u^\mathfrak{e})\star {\pi}_1(y)) & \;\;\;\textrm{ if $\;\;t=\l xu$},
\\\l x(e^\mathfrak{e}\star t^\mathfrak{e})& \;\;\;\textrm{ if $\;\;t=\mut
x\lfloor t , e\rfloor$},\\u^\mathfrak{e}&\;\;\;\textrm{ if
$\;\;t=\overline{u}$}.\end{array}\right.
\]
\[ e^\mathfrak{e}=\left\{ \begin{array}{ll} \overline{\a} & \;\;\;\textrm{ if $\;\;e=\a$},\\
\lan t^\mathfrak{e} , h^\mathfrak{e}\ran & \;\;\;\textrm{ if $\;\;e=t.h$},
\\  \l \overline{\a}(e^\mathfrak{e}\star t^\mathfrak{e})& \;\;\;\textrm{ if $\;\;e=\mu \a\lfl t , e\rfl$},
\\h^\mathfrak{e}&\;\;\;\textrm{ if
$\;\;e=\widetilde{h}$}.
\end{array}\right. \]
\item The translation $.^\mathfrak{e}$ also applies to types.
\begin{itemize}
\item $A^\mathfrak{e}=A$, where $A$ is an atomic type,
\item $(A^\bot)^\mathfrak{e}=(A^\mathfrak{e})^\bot$,
\item $(A\ra B)^\mathfrak{e}=(A^\mathfrak{e})^\bot\vee B^\mathfrak{e}$.
\end{itemize}
\item Let $\G$, $\D$ be $l$- and $r$-contexts, respectively. Then
$\G^\mathfrak{e}=\{x:A^\mathfrak{e}\; | \; x:A\in
\G\}$ and similarly for $\D$. Furthermore, for any $r$-context $\D$, let $\D^{\bot}=\{\overline{\a}: A^\bot \;|\; \a:A\in \D\}$.
\end{enumerate}
\end{defi}

\begin{lem} \label{ch5:typelmtsls}
\begin{enumerate}

\item If $\G \;{\vv} \; t : A\; |\;  \D$, then $\G^\mathfrak{e} ,
(\D^\mathfrak{e})^{\bot}\;{\v} \; t^\mathfrak{e}:A^\mathfrak{e}$.
\item If $\G \; |\; e : A\;{\vv} \;  \D$, then $\G^\mathfrak{e} ,
(\D^\mathfrak{e})^{\bot}\;{\v} \; e^\mathfrak{e}: (A^\mathfrak{e})^\bot$.
\item If $p:(\G\;{\vv} \;\D )$, then $\G^\mathfrak{e},
(\D^\mathfrak{e})^{\bot}\;{\v} \;p^\mathfrak{e}:\bot$.
\end{enumerate}
\end{lem}

\begin{proof}
The above statements are proved simultaneously according to the
length of the $\lmts$-deduction. We remark that $.^{\mathfrak{e}}$ is defined in Definition \ref{ch5:pi} exactly
in the way to make the assertions of the lemma true.
Let us examine some of the more interesting cases.
\begin{enumerate}

\item Suppose
\[
\frac{\G , x:A\;{\vv}\;u:B\;|\;\D}{\G \;{\vv}\;\l xu:A\ra B\;|\;\D}.
\]
Then we have, by the induction hypothesis and Notation \ref{ch5:pi},
\[\G^\mathfrak{e} , (\D^\mathfrak{e})^\bot , x:A^\mathfrak{e} ,
y:A^\mathfrak{e}\wedge (B^\mathfrak{e})^\bot
\;{\v}\;u^\mathfrak{e}:B^\mathfrak{e},\]
\[\G^\mathfrak{e} , (\D^\mathfrak{e})^\bot ,
y:A^\mathfrak{e}\wedge (B^\mathfrak{e})^\bot
\;{\v}\;{\pi}_1(y):A^\mathfrak{e},\]
\[\G^\mathfrak{e} , (\D^\mathfrak{e})^\bot ,
y:A^\mathfrak{e}\wedge (B^\mathfrak{e})^\bot
\;{\v}\;{\pi}_2(y):(B^\mathfrak{e})^\bot.\]
Thus we can conclude
\[\G^\mathfrak{e} , (\D^\mathfrak{e})^\bot , x:A^\mathfrak{e} ,
y:A^\mathfrak{e}\wedge (B^\mathfrak{e})^\bot
\;{\v}\;({\pi}_2(y)\star u^\mathfrak{e}):\bot,
\]
\[\G^\mathfrak{e} , (\D^\mathfrak{e})^\bot ,
y:A^\mathfrak{e}\wedge (B^\mathfrak{e})^\bot\;{\v}\;\l
x({\pi}_2(y)\star u^\mathfrak{e}): (A^\mathfrak{e})^\bot.
\]
From which, we obtain
\[\G^\mathfrak{e} , (\D^\mathfrak{e})^\bot\;{\v}\;
\l y(\l x({\pi}_2(y)\star u^\mathfrak{e})\star {\pi}_1(y)):
(A^\mathfrak{e})^\bot\vee B^\mathfrak{e}.
\]
\item Assume now
\[\frac{\G\;{\vv}\;t:A\;|\;\D \;\;\;\;\; \G\;|\;e:B\;{\vv}\;\D}
{\G\;|\;t.e:A\ra B\;{\vv}\;\D}.
\]
Then we have
\[
\frac{\G^\mathfrak{e} , (\D^\mathfrak{e})^\bot\;{\v}\;t^\mathfrak{e}:A^\mathfrak{e} \;\;\;\;\;
\G^\mathfrak{e} , (\D^\mathfrak{e})^\bot\;{\v}\;e^\mathfrak{e}:(B^\mathfrak{e})^\bot}
{\G^\mathfrak{e} , (\D^\mathfrak{e})^\bot\;{\v}\;
\lan t^\mathfrak{e} , e^\mathfrak{e}\ran :A^\mathfrak{e}\wedge (B^\mathfrak{e})^\bot}.
\]
\item From
\[
\frac{\G\;{\vv}\;t:A\;|\;\D   \;\;\;\;\;  \G\;|\;e:A\;{\vv}\;\D}
{\lfl t , e\rfl :(\G\;{\vv}\;\D)},
\]
we obtain
\[
\frac{\G^\mathfrak{e} , (\D^\mathfrak{e})^\bot\;{\v}\;t^\mathfrak{e}:A^\mathfrak{e} \;\;\;\;\;
\G^\mathfrak{e} , (\D^\mathfrak{e})^\bot\;{\v}\;e^\mathfrak{e}:(A^\mathfrak{e})^\bot}
{\G^\mathfrak{e} , (\D^\mathfrak{e})^\bot\;{\v}\;
(e^\mathfrak{e}\star t^\mathfrak{e}):\bot}.
\]
\end{enumerate}
\end{proof}

\noindent Our next aim is to prove that $\lmts$ can be simulated by the
\smash{$\ls$}-calculus. To this end we introduce a new notion of equality
in the $\ls$-calculus.

\begin{defi}
We define an equivalence relation $\sim$ on $\ma{T}$, which is the smallest relation compatible with the term forming rules and containing $((M\star N),(N\star M))$.
\begin{itemize}
\item $x\sim x$,
\item if $M\sim M'$, then $\l xM\sim \l xM'$ and $\si_i(M)\sim \si_i(M')$ for $i\in \{1, 2\}$,
\item if $M\sim M'$ and $N\sim N'$, then $\lan M , N\ran \sim \lan M' , N'\ran$ and $(M \star N)\sim (M'\star N')$
and $(M \star N)\sim (N'\star M')$.
\end{itemize}
We say that $M$ and $N$ are equal up to symmetry provided $M \sim N$.
\end{defi}

\begin{lem} \label{ch4:sbstsim} \label{ch4:ppsim}
Let $M,M',N,N' \in \ma{T}$.
\begin{enumerate}

\item If $M\sim M'$ and $N\sim N'$, then $M[x:=N]\sim M'[x:=N']$.
\item If $M\sim M'$ and $M'\ra N$, then there is $N'$ for which $M\ra N'$ and $N\sim N'$.
\end{enumerate}
\end{lem}

\begin{proof}
1. By induction on $cxty(M)$. 2. By 1.
\end{proof}

\begin{lem} \label{ch5:subls}
Let $u,t,e \in \mathfrak{T}$. Then $(u[x:=t])^\mathfrak{e}=u^\mathfrak{e}[x:=t^\mathfrak{e}]$ and
$(u[a:=e])^\mathfrak{e}=u^\mathfrak{e}[a:=e^\mathfrak{e}]$.
\end{lem}

\begin{proof}
By induction on $cxty(u)$.
\end{proof}

Now we can formulate our assertion about the simulation of the $\lmts$-calculus by the $\ls$-calculus.

\begin{thm} \label{ch5:simlmtsls}~
Let $v,w \in \mathfrak{T}$.
\begin{enumerate}

\item If $v\raa_r w$ and $r \in
\{ \b \, , \, \mu \, , \, {\mut} \, , \, {s_l} \, , \, {s_r} \}$, then $\;v^\mathfrak{e}\ra^+w^\mathfrak{e}$.
\item If $v\raa_r w$ and $r\in\{{cl_{1,l}} \, , \, {cl_{1,r}} \, , \,  {cl_2} \}$, then  $v^\mathfrak{e}\sim w^\mathfrak{e}$.
\end{enumerate}
\end{thm}

\begin{proof}
\begin{enumerate}

\item Let us only treat the typical cases.
\begin{enumerate}

\item If $v=\lfl \l xu , (t.e)\rfl \raa_{\b} \lfl t , \mut x\lfl u ,
e\rfl \rfl =w$, then
$v^\mathfrak{e} = (\lan t^\mathfrak{e} , e^\mathfrak{e}\ran \star
\l y(\l x({\pi}_2(y)\star
u^\mathfrak{e})\star {\pi}_1(y)))\ra_{\b^{\bot}}
{} (\l x({\pi}_2(\lan t^\mathfrak{e} , e^\mathfrak{e}\ran )\star
u^\mathfrak{e})\star
{\pi}_1(\lan t^\mathfrak{e} , e^\mathfrak{e}\ran ))\ra^*
{} (\l x(e^\mathfrak{e}\star u^\mathfrak{e})\star
t^\mathfrak{e})=w^\mathfrak{e}$.

\item If $v=\lfl \mu ap , e\rfl \raa_{\m} p[a:=e]=w$, then, by Lemma
\ref{ch5:subls},
$v^\mathfrak{e}=(e^\mathfrak{e}\star \l ap^\mathfrak{e})
\ra_{\b^{\bot}}p^\mathfrak{e}[a=e^\mathfrak{e}]=w^\mathfrak{e}$.
\item If $v=\mu a\lfl w , a\rfl \raa_{s_l} w$, $a\notin w$, then
$v^\mathfrak{e}=\l a (a\star w^\mathfrak{e})\ra_{\eta^\bot}w^\mathfrak{e}$.
\end{enumerate}
\item
\begin{enumerate}

\item If $v=\overline{\widetilde{u}}\raa_{cl_{1,l}}u=w$, then $v^\mathfrak{e}=(\overline{\widetilde{u}})^\mathfrak{e}=
u^\mathfrak{e}=w^\mathfrak{e}$.
\item If $v=\lfl \overline{v} , \widetilde{u}\rfl \raa_{cl_2} \lfl u , v\rfl =w$, then
$v^\mathfrak{e}=\lfl \overline{v} , \widetilde{u}\rfl^\mathfrak{e}
=(u^{\mf{e}}\star v^{\mf{e}})\sim w^\mathfrak{e}$.\qedhere
\end{enumerate}
\end{enumerate}
\end{proof}

\begin{cor}
The $\lmts$-calculus is strongly normalizable.
\end{cor}

\begin{proof}
Let $\si$ be a reduction sequence in the $\lmts$-calculus, assume $\si$ is $v_0\raa v_1\ldots \raa v_n$ and $\si$ contains $k\geq 0$ number of $\b$-, $\m$-, $\mut$-, $s_l$- or $s_r$-reductions. By Theorem \ref{ch5:simlmtsls}, $v_0^\mathfrak{e}$, $v_1^\mathfrak{e},\ldots$, $v_n^\mathfrak{e}$ forms a sequence of \smash{$\ls$}-terms, where either $v_i^\mathfrak{e}\ra v_{i+1}^\mathfrak{e}$ or $v_i^\mathfrak{e}\sim v_{i+1}^\mathfrak{e}$ $(0\leq i\leq n-1)$ and, for every $\b$-, $\m$-, $\mut$-, $s_l$- or $s_r$-reduction, there corresponds a reduction step in the \smash{$\ls$}-calculus. By Lemma \ref{ch4:ppsim}, we obtain that $\sim$ can be postponed, that is, there are $w_0$, $w_1,\ldots$, $w_{k+1}$ in $\mathcal{T}$ such that $w_0=v_0^\mathfrak{e}$, $w_{k+1}=v_n^\mathfrak{e}$ and $w_0\ra\ldots \ra w_k\sim w_{k+1}$. This means that we can establish a reduction sequence of length $k$ starting from $v_0^\mathfrak{e}$ in the \smash{$\ls$}-calculus. Hence, by Theorem \ref{ch3:snt} and Corollary \ref{ch3:snlsfull}, an 
infinite reduction sequence starting from $v_0$ can contain only a finite number of $\b$-, $\m$-, $\mut$-, $s_l$- or $s_r$-reductions. Thus there would exist an infinite reduction sequence in the $\lmts$-calculus consisting entirely of $cl_{1,l}$-, $cl_{1,r}$- and $cl_2$-reductions, which is impossible.
\end{proof}

\subsection{A translation of the \texorpdfstring{$\ls$}{lambda-Sym-Prop}-calculus into the \texorpdfstring{$\lmts$}{blambda-mu-bmu}-calculus}

Now we are going to deal with the converse relation. That is we
will present a translation of the $\ls$-calculus into the $\lmts$-calculus which faithfully
reflects the typability relations of one calculus in the other
one. Then we prove that our translation is in fact a simulation of the
$\ls$-calculus in the $\lmts$-calculus.

\begin{defi}\label{ch5:mumutilde}\hfill
\begin{enumerate}

\item The translation $.^\mathfrak{f}: \ma{T} \longrightarrow \mathfrak{T} $ is defined as follows.
\[ M^{\mathfrak{f}}=\left\{ \begin{array}{ll} x & \;\;\;\textrm{ if $\;\;M=x$},\\
\lfl Q^\mathfrak{f} , \widetilde{P^\mathfrak{f}}\rfl & \;\;\;
\textrm{ if $\;\;M=(P\star Q)$},\\
\overline{\mut xN^\mathfrak{f}}& \;\;\;\textrm{ if $\;\;M=\l xN$},\\
\overline{(P^\mathfrak{f} .\widetilde{Q^\mathfrak{f}})}&\;\;\;\textrm{ if $\;\;M=\lan P ,
Q\ran$},\\
\l x\mu \b\lfl N^\mathfrak{f} , \widetilde{x}\rfl & \;\;\;\textrm{ if $\;\;M=\si_1(N)$,
$x\notin Fv(N^\mathfrak{f})$ and $\b\notin Fv(\lfl N^\mathfrak{f} , \widetilde{x}\rfl)$},\\
\l xN^\mathfrak{f} & \;\;\;\textrm{ if $\;\;M=\si_2(N)$ and $x\notin Fv(N^\mathfrak{f})$}.
\end{array}\right. \]
\item The translation $.^\mathfrak{f}$ applies to the types as follows.
\begin{itemize}
\item ${\a}^\mathfrak{f}=\a$,
\item ${({\a}^\bot )}^\mathfrak{f}={\a}^\bot$,
\item $(A\wedge B)^\mathfrak{f} =(A^\mathfrak{f} \ra (B^\mathfrak{f})^\bot )^\bot$,
\item $(A\vee B)^\mathfrak{f} = (A^\mathfrak{f})^\bot \ra B^\mathfrak{f}$.
\end{itemize}
We remark that $.^\mf{f}$ maps the terms of the $\ls$-calculus
with type $\bot$ to $c$-terms of the $\lmts$-calculus, which have
no types. We also have, for all types $A$, ${(A^\bot )}^\mathfrak{f} =
(A^\mathfrak{f})^\bot$. Therefore the translation $.^\mf{f}$ maps equal types to equal types.
\end{enumerate}
\end{defi}

\begin{lem} \label{ch5:typelslmts}
\begin{enumerate}

\item If $\G \;{\v} \; M : A$ and $A\neq \bot$, then $\G^\mathfrak{f}\;{\vv} \; M^\mathfrak{f} :A^\mathfrak{f}$.
\item If $\G \;{\v} \; M : \bot$, then  $M^\mathfrak{f} :(\G^\mathfrak{f}\;{\vv} \;)$.
\end{enumerate}
\end{lem}

\begin{proof} The proof proceeds by a simultaneous induction on the length of the
derivation in the $\ls$-calculus. We can observe again that the notion of $.^\mathfrak{f}$ in Definition \ref{ch5:mumutilde}
is conceived in a way to make the statements of the lemma true.
Let us only examine some of the typical cases of the first assertion.
\begin{enumerate}

\item Suppose
\[
\frac{\G , x : A\;{\v}\; u : \bot}
{\G\;\v\;\l xu : A^\bot}.
\]
Then, applying the induction hypothesis,
\[
\dfrac{
\dfrac{u^\mathfrak{f} :(\G^\mathfrak{f} , x : A^\mathfrak{f} \;{\vv}\;)}
{\G^\mathfrak{f} \;|\;\mut xu^\mathfrak{f} : A^\mathfrak{f} \;\vv\;}}
{\G^\mathfrak{f} \;\vv\;\overline{\mut xu^\mathfrak{f} }:(A^\mathfrak{f})^\bot}.
\]
\item
If
\[
\frac{\G \;{\v}\; u:A}
{\G \;{\v}\; \si_1(u):A\vee B},
\]
then, we obtain
\[
\dfrac{
\dfrac{
\dfrac{
\dfrac{\;}
{\G^\mathfrak{f} , x:(A^\mathfrak{f})^\bot \;{\vv}\;u^\mathfrak{f} :
A^\mathfrak{f} \;|\;\b :B^\mathfrak{f}} \;\;\;\;\;
\dfrac{\G^\mathfrak{f} , x:(A^\mathfrak{f})^\bot \;{\vv}\;x:(A^\mathfrak{f})^\bot
\;|\;\b :B^\mathfrak{f}}
{\G^\mathfrak{f} , x:(A^\mathfrak{f})^\bot \;|\;\widetilde{x}:
A^\mathfrak{f} \;{\vv}\;\b : B^\mathfrak{f}}}
{\lfl u^\mathfrak{f} , \widetilde{x}\rfl :
(\G^\mathfrak{f} , x:(A^\mathfrak{f})^\bot \;{\vv}\;\b :B^\mathfrak{f})}}
{\G^\mathfrak{f} , x:(A^\mathfrak{f})^\bot \;{\vv}\;
\mu \b \lfl u^\mathfrak{f} , \widetilde{x}\rfl :B^\mathfrak{f}}}
{\G^\mathfrak{f} \;{\vv}\;\l x\mu \b \lfl u^\mathfrak{f} , \widetilde{x}\rfl :
(A^\mathfrak{f})^\bot \ra B^\mathfrak{f}}.
\]

\item From
\[
\frac{\G \;{\v}\; u : A^\bot \;\;\;\;\; \G \;{\v}\; v: A}
{\G \;{\v}\; (u\star v):\bot},
\]
we obtain
\[
\dfrac{
\dfrac{\G^\mathfrak{f} \;{\vv}\; u^\mathfrak{f}:(A^\mathfrak{f})^\bot}
{\G^\mathfrak{f} \;|\; \widetilde{u^\mathfrak{f}}:A^\mathfrak{f}\;{\vv}\;} \;\;\;\;\;\;
\dfrac{\;}
{\G^\mathfrak{f} \;{\vv}\; v^\mathfrak{f}:A^\mathfrak{f}}}
{\lfl v^\mathfrak{f} , \widetilde{u^\mathfrak{f}}\rfl :(\G^\mathfrak{f} \;{\vv}\;)}.
\]
\end{enumerate}
\end{proof}

\noindent Now we turn to the proof of the simulation of the $\ls$-calculus in the $\lmts$-calculus.

\begin{lem} \label{ch5:sublb}
Let  $M,N \in \ma{T}$. Then ${(M[x:=N])}^\mathfrak{f} =M^\mathfrak{f} [x:=N^\mathfrak{f}]$.
\end{lem}

\begin{proof}
By induction on $cxty(M)$.
\end{proof}

\begin{thm} \label{ch5:simlslmts}
Let $M,N \in \ma{T}$. If $M \ra N$, then $M^\mathfrak{f} \raa^+N^\mathfrak{f}$.
\end{thm}

\begin{proof} Let us prove some of the more interesting cases.
\begin{enumerate} 

\item If $M=(\l xP\star Q)\ra_{\b}P[x:=Q]=N$, then, applying Lemma \ref{ch5:sublb},

$M^\mathfrak{f} = \lfl Q^\mathfrak{f} ,
\widetilde{\overline{\mut xP^\mathfrak{f} }}\rfl
\raa_{cl_{1}} {}\lfl
Q^\mathfrak{f} , \mut xP^\mathfrak{f} \rfl
\raa_{\mut} {} P^\mathfrak{f} [x:=Q^\mathfrak{f} ]=N^\mathfrak{f}$.

\item If $M=(Q\star \l xP)\ra_{\b_{\bot}}P[x:=Q]=N$, then

$M^\mathfrak{f} = \lfl \overline{\mut xP^\mathfrak{f} } ,
\widetilde{Q^\mathfrak{f}}\rfl \raa_{cl_{2}} {} \lfl Q^\mathfrak{f}, \mut
xP^\mathfrak{f} \rfl
\raa_{\mut} {} P^\mathfrak{f} [x:=Q^\mathfrak{f} ]=N^\mathfrak{f}$.

\item If $M=(\lan P , Q\ran \star \si_1(R))\ra_{\pi}(P\star R)=N$, then

$M^\mathfrak{f} = \lfl \l x\mu b\lfl R^\mathfrak{f} ,
\widetilde{x}\rfl ,
\widetilde{\overline{(P^\mathfrak{f} .\widetilde{Q^\mathfrak{f}})}}\rfl \raa_{cl_1}
{} \lfl \l x\mu b\lfl R^\mathfrak{f} , \widetilde{x}\rfl ,
(P^\mathfrak{f}
.\widetilde{Q^\mathfrak{f}},)\rfl \raa_{\l}$

$\lfl P^\mathfrak{f} , \mut x\lfl \mu b\lfl R^\mathfrak{f} ,
\widetilde{x}\rfl , \widetilde{Q^\mathfrak{f}}\rfl \rfl \raa_{\mut}
{} \lfl \mu b\lfl R^\mathfrak{f} ,
\widetilde{P^\mathfrak{f}}\rfl , \widetilde{Q^\mathfrak{f}}\rfl \raa_{\mu}
{} \lfl R^\mathfrak{f} , \widetilde{P^\mathfrak{f}}\rfl =
N^\mathfrak{f}$.\qedhere
\end{enumerate}
\end{proof}

We could have as well demonstrated that the $\lmts$-calculus is strongly normalizable by applying
the method presented in Section 3 as accomplished by Batty\'anyi \cite{Batt}. The following result states that in this
case the strong normalizability of the \smash{$\ls$}-calculus would arise as a direct consequence
of that of the $\lmts$-calculus.

\begin{cor}
If the $\lmts$-calculus is strongly normalizable, then the same is
true for the $\ls$-calculus as well.
\end{cor}

\begin{proof}
By Theorem $\ref{ch5:simlslmts}$.
\end{proof}

\subsection{The connection between the two translations}

In this subsection we  examine the connection between the two transformations.
We prove that both compositions
$.^{\mf{e}^\mf{f}}:\mf{T} \longrightarrow \mf{T}$ and $.^{\mf{f}^\mf{e}}:\ma{T} \longrightarrow  \ma{T}$
 are such that we can get back the original terms by performing some steps of reduction
on $u^{\mf{e}^\mf{f}}$ or on $M^{\mf{f}^\mf{e}}$, respectively. That is, the following
theorems are valid. The case of $.^{\mf{f}^\mf{e}}$ is the easier one.

First we describe the effect of $.^{\mf{f}^\mf{e}}$ on the typing relations.

\begin{lem}
 If $\G \;{\vdash} \; M : A$, then $\G\;{\vdash} \; {M^\mathfrak{f}}^\mathfrak{e}:A$.
\end{lem}

\begin{proof}
Combining Lemmas \ref{ch5:typelslmts} and \ref{ch5:typelmtsls}.
\end{proof}

\begin{thm}\label{comp:fe}
Let $M \in \ma{T}$. Then ${M^{\mf{f}}}^\mf{e} \ra^* M$.
\end{thm}

\begin{proof}
By induction on $cxty(M)$. We consider only the more interesting cases.
\begin{enumerate}

\item  If $M=(P\star Q)$, then
${(P\star Q)^\mf{f}}^\mf{e}=\lfl Q^\mf{f},\widetilde{P^\mf{f}}\rfl^\mf{e}=({P^\mf{f}}^\mf{e}\star{Q^\mf{f}}^\mf{e})\ra^*(P\star Q)$.
\item If $M=\lan P,Q\ran$, then
${\lan P,Q\ran^\mf{f}}^\mf{e}=\overline{({P^\mf{f}.\widetilde{Q^\mf{f}}})}^\mf{e}=\lan {P^\mf{f}}^\mf{e},{Q^\mf{f}}^\mf{e}\ran\ra^*
\lan P,Q\ran$.
\item If $M=\si_1(N)$, then

${\si_1(N)^\mf{f}}^\mf{e}={\l x\mu \b\lfl N^\mf{f},\widetilde{x}\rfl}^\mf{e}=\l y(\l x({\pi}_2(y)\star {(\mu \b\lfl N^\mf{f},
\widetilde{x}\rfl)}^\mathfrak{e})\star {\pi}_1(y))=\\
\l y(\l x({\pi}_2(y)\star \l \overline{\b}(x\star {N^\mf{f}}^\mathfrak{e}))\star {\pi}_1(y))\ra_{\b^\bot}\l y
(\l x(x\star {N^\mf{f}}^\mathfrak{e})\star {\pi}_1(y))\ra_{\eta^\bot}\\ \l y({N^\mf{f}}^\mathfrak{e}\star {\pi}_1(y))
\ra_{\b}\l y(y\star \si_1({N^\mf{f}}^\mf{e}))\ra_{\eta^\bot}
\si_1({N^\mf{f}}^\mf{e})\ra^*\si_1(N)$.\qedhere
\end{enumerate}
\end{proof}

\noindent We begin to examine the composition ${.^\mf{e}}^\mf{f}:\mf{T}\ra\mf{T}$ for an arbitrary $u$.
First we make the following observation.

\begin{lem}
\begin{enumerate}

\item If $\G \;{\vv} \; t : A\; |\;  \D$, then $\G,
\D^\bot\;{\vv} \; {t^\mathfrak{e}}^\mathfrak{f}:A$.
\item If $\G \; |\; e : A\;{\vv} \;  \D$, then $\G,\D^\bot\;{\vv} \; {e^\mathfrak{e}}^\mathfrak{f}: A^\bot$.
\item If $p:(\G\;{\vv} \;\D )$, then ${p^\mathfrak{e}}^\mathfrak{f}:(\G,\D^\bot\;{\vv} )$.
\end{enumerate}
\end{lem}

\begin{proof}
Combining Lemmas \ref{ch5:typelmtsls} and \ref{ch5:typelslmts}.
\end{proof}

Theorem \ref{comp:fe} states that, if $M$ is an \smash{$\ls$}-term, then $M$ can be related to ${M^{\mf{f}}}^\mf{e}$
by the reductions in the \smash{$\ls$}-calculus. We note that we are not able to obtain $u$ from ${{u^{\mf{e}}}}^\mf{f}$
in such a way. We can find a term $T$ instead such that ${{u^{\mf{e}}}}^\mf{f}\raa^* T(u)$. The function $T$ can intuitively
be considered as the description how \smash{$\ls$}-connectives can be embedded into the $\lmts$-calculus.
It turns out that the $\lmts$-calculus translates the \smash{${\ls}$}-terms not so smoothly as it was the case with the
other direction.

\begin{defi}\label{comp:T}
We define a function $T$ assigning a $\lmts$-term to a
$\lmts$-term.\\\\
\begin{tabular}{ll}
\begin{minipage}[t]{220pt}
\begin{itemize}
\item $T(x)=x$,
\item $T(\l xu)=\overline{\mut y\lfl T(u)[x:=p_1(y)] , \wt{p_2(y)}\rfl}$,
\item  $T(\m \a p)=\overline{\mut \overline{\a}T(p)}$,
\item $T(\overline{u})=T(u)$,
\end{itemize}
\end{minipage}
&
\begin{minipage}[t]{150pt}
\begin{itemize}
\item  $T(\a)=\overline{\a}$,
\item  $T((u.v))=\lan T(u) , T(v)\ran$,
\item $T(\mut xp)=\overline{\mut xT(p)}$,
\item  $T(\wt{h})=T(h)$,
\item  $T(\lfl t ,e\rfl)=\lfl T(t) ,\wt{T(e)}\rfl$.
\end{itemize}
\end{minipage}
\end{tabular}
\end{defi}

\begin{thm}
Let $u \in \mathfrak{T}$. We have ${u^{\mf{e}}}^\mf{f}\raa^* T(u)$.
\end{thm}

\begin{proof}
By induction on $cxty(u)$. We consider only some of the cases.
\begin{enumerate}

\item If $u=\l xv$, then

${u^{\mf{e}}}^\mf{f} = (\l y(\l x({\pi}_2(y)\star u^\mathfrak{e})\star {\pi}_1(y)))^\mf{f}
= \overline{\mut y\lfl p_1(y),\widetilde{\overline{\mut x\lfl {v^{\mf{e}}}^\mf{f},\widetilde{p_2(y)}\rfl}}\rfl} \raa_{cl_{1,r}}$ \\
$\overline{\mut y\lfl p_1(y),\mut x\lfl {v^{\mf{e}}}^\mf{f},\widetilde{p_2(y)}\rfl\rfl}
\raa_{\mut} \overline{\mut y\lfl {v^{\mf{e}}}^\mf{f}[x:=p_1(y)],\widetilde{p_2(y)}[x:=p_1(y)]\rfl}
\raa^* T(u)$.\medskip

\item If $u=\mut x\lfl t,v\rfl$, then

${u^{\mf{e}}}^\mf{f} = (\l x(v^\mf{e}\star t^\mf{e}))^\mf{f}
=\overline{\mut x\lfl {t^\mf{e}}^\mf{f},\widetilde{{v^\mf{e}}^\mf{f}}\rfl}
\raa^* \overline{\mut x\lfl T(t),\widetilde{T(v)}\rfl}
=\overline{\mut xT(\lfl t,v\rfl)}=T(u)$.\qedhere
\end{enumerate}
\end{proof}

\begin{rem}
We remark that we cannot expect $T(u)$ to be expressible with the help of $\mf{T}$. Namely, we can show that, if
\smash{$=_{\tiny{\lmts}}$} denotes the reflexive, transitive closure of the compatible union of the reduction relations in
the $\lmts$-calculus, then none of the assertions below are valid.
\begin{enumerate}

\item There exists a a $\lmts$-term $\Phi$ such that, for every $c$-term $c$, $T(c)=_{\tiny{\lmts}}\Phi(c)$.
\item There exists a a $\lmts$-term $\Phi_1$ such that, for every $l$-term $t$, $T(t)=_{\tiny{\lmts}}\Phi_1(t)$.
\item There exists a a $\lmts$-term $\Phi_2$ such that, for every $r$-term $e$, $T(e)=_{\tiny{\lmts}}\Phi_2(e)$.
\end{enumerate}
\end{rem}

\section{Conclusion}

The paper is mainly devoted to an arithmetical proof of the strong normalization of the \smash{$\ls$}-calculus introduced by
Berardi and Barbanera \cite{Ber-Bar}. The proof is an adaptation of the work of David and Nour \cite{Dav-Nou3}. The novelty of our
 paper is the application of the method of zoom-in sequences of redexes: we achieve the main theorem by identifying the
minimal non-strongly normalizing redexes of an infinite reduction sequence, which we call a zoom-in sequence of redexes.
The idea of zoom-in sequences was inspired by the notion of perpetual reduction strategies introduced by Raamsdonk et al. \cite{Sor}. Following the proof of the strong normalization of the \smash{$\ls$}-calculus, the $\lmt$-calculus is introduced, which was defined
by Curien and Herbelin \cite{Cur-Her}. The same proof of strong normalization as we have presented for the \smash{$\ls$}-calculus would also
work for the calculus of Curien and Herbelin as was shown by Batty\'anyi \cite{Batt}. However, instead of adapting the proof method for the
$\lmt$-calculus, we designed a translation of the \smash{$\ls$}-calculus in the $\lmts$-calculus and vice versa,
where the $\lmts$-calculus is the $\lmt$-calculus augmented with terms explicitly expressing negation and with
rules handling them. The translation allows us to assert strong normalization for the $\lmts$- and, hence, for the $\lmt$-calculus.

On the technical side, we remark that there were two main difficulties that rendered the proof a little more involved. First,
we had to work with an alternating substitution defined inductively starting from two sets of terms. The reason was that we had
to prove a more general statement to locate the supposedly non strongly normalizing part of a term emerging as a result of a
substitution. Simple substitutions would not have been enough for our purpose. The second difficulty was that in order
to establish a key property of zoom-in sequences in Lemma \ref{ch3:zoom} we had to move forward the Hypothesis ``H'' from the
main theorem, thus making the application of the hypothesis implicit in the sequel. We think that the elimination of
both problems would considerably enhance the paper's intelligibility.

It seems promising to investigate whether the present method of verifying strong normalization can be applied to systems other
than simple typed logical calculi, for example, proof nets (Laurent \cite{Lau}). Another fields of interest could be intuitionistic
and classical typed systems with explicit substitutions (Rose \cite{Ros}). To handle these systems, the present proof must be
simplified, we have to pay attention in our proof, for example, that the substitutions are defined by two sets of terms of
different types. Finally, we remark that it is a natural requirement of a proof formalizable in first order arithmetic to
enable us to find an upper bound for the lengths of the reduction sequences. At its present form, our proof does not make
it possible, this raises another demand for the simplification of the results.

\section*{Acknowledgment}
  \noindent We wish to thank Ren\'e David and the anonymous referees for helpful discussions and remarks.


\end{document}